%% file: main.tex
\documentclass[a4paper]{amsart}
\input{preamble}
\input{environments}
\input{macros}
\title[Seifert surfaces and composition of binary quadratic forms]{Seifert surfaces in the four-ball and composition of binary quadratic forms}
\author{Menny Aka}
\author{Peter Feller}
\author{Alison Beth Miller}
\author{Andreas Wieser}

\address[M.A.]{Department of Mathematics, ETH Z\"urich, R\"amistrasse 101, 8092 Z\"urich, Switzerland}
\email{menny.akka@math.ethz.ch}

\address[P.F.]{Université de Neuchâtel, 
Rue Emile-Argand 11, 2000 Neuchâtel, Switzerland}
\email{peter.feller@unine.ch}

\address[A.B.M.]{Mathematical Reviews, 535 W William St, Ste 210, Ann Arbor, MI 48103, USA}
\email{alimil@umich.edu}
\address[A.W.]{Department of Mathematics, Princeton University, Fine Hall, Washington Road, Princeton, NJ 08540, USA}
\email{aw2638@math.princeton.edu}

\thanks{PF~was supported by the SNSF grant 181199. AW was funded by ERC grant HomDyn, ID~833423.}

\subjclass[2020]{57K10, 57K45, 11E16, 15A72}
\keywords{Gauss composition, 
Klein vectors, Seifert surfaces, Seifert form}
\begin{document}
\begin{abstract} 
We use composition of binary quadratic forms to systematically create pairs of Seifert surfaces that are non-isotopic in the four-ball. 
Our main topological result employs Gauss composition to classify the pairs of binary quadratic forms that arise as the Seifert forms of pairs of disjoint Seifert surfaces of genus one.
The main ingredient of the proof is number-theoretic and of independent interest.
It establishes a geometric approach to Gauss composition via planes in the space of two-by-two matrices. 
\end{abstract}
\maketitle

\section{Introduction}\label{sec:introduction} 
In this article, we revisit Gauss composition of binary quadratic forms via planes in four-dimensional space. We use this perspective to develop a robust construction of Seifert surfaces that are non-isotopic in the four-ball.
We begin here by discussing these applications in \S\ref{sec:topological results} but note that the number-theoretic result in \S\ref{sec:numbertheoretic results} is of independent interest. 
In fact, the number-theoretic parts of the paper (\S\ref{sec:numbertheoretic results}, \S\ref{sec:Intro:sympplanes} in the introduction, and \S\ref{sec:notation}, \S\ref{sec:Kleinmap}, \S\ref{sec:Kleinrelations} in the body of the text) can be read completely independently of the low-dimensional topology part of the paper. 
The topological part of the paper is written to be independent but makes reference to some notation and results of the number-theoretic parts.

\subsection{Seifert surfaces with the same boundary}\label{sec:topological results}
Recently, Hayden, Kim, Miller, Park, and Sundberg~\cite{HMKPS22} found a pair of genus one Seifert surfaces $F_1$ and $F_2$ for the same knot $K$ that, after pushing their interior into the four-ball $\B^4$, are not topologically ambient isotopic in $\B^4$. 
Here, recall that a Seifert surface for a knot $K$ is a compact oriented smooth surface in the $3$-sphere $S^3$ with boundary $K$.
This resolved a question by Livingston~\cite{Livingston82}, who asked whether two such Seifert surfaces can exist.
The example in \cite{HMKPS22} is such that the interiors are disjoint; we call pairs of genus 1 Seifert surfaces with the same boundary and disjoint interior \emph{disjoint genus one pairs} from here on.
In this paper, we put the recent example into a more general context by producing algebraic characterizations of when such pairs of Seifert surfaces with prescribed classical invariants (derived from the Seifert form) exist. 
The key result that makes all our topology results possible is a complete number-theoretic characterization of those pairs of Seifert forms that arise from disjoint genus one pairs of Seifert surfaces. 

It turns out that this characterization is in terms of Gauss composition for integral binary quadratic forms~\cite{gauss}.
To state our characterization, we recall some background on integral binary quadratic forms and knot theory.  An \emph{integral binary quadratic form} (\emph{IBQF}) is an expression of the form $ax^2+bxy+cy^2$, $a, b, c \in \Z$, and its \emph{discriminant} is $b^2-4ac$.  There is a natural action of $\SL_2(\Z)$ on IBQFs by change of basis, which preserves the discriminant.
Going back to Gauss, there is a composition law for classes of IBQFs of equal discriminant, which we denote by $*$ and recall in \S\ref{sec:notation}.

For an oriented compact smooth surface with boundary $F\subset S^3$, the Seifert form $U_F\colon H_1(F;\Z)\times H_1(F;\Z)\to \Z$ is the (non-symmetric) bilinear form defined by $U_F([\alpha],[\beta])=\mathrm{lk}(\alpha,\beta^+)$ where $\mathrm{lk}$ denotes Gauss's linking number of disjoint curves in $S^3$, $\alpha$ and $\beta$ are curves in $F$, and $\beta^+$ is the result of pushing $\beta$ a small amount along a positive normal into $S^3\setminus F$.

In the case of interest to us, $F$ is a genus $1$ surface with boundary a knot $K$, and so $H_1(F; \Z)$ has rank 2.  In this case the information contained in the Seifert pairing $U_F$ can equivalently be packaged in the binary quadratic form
\begin{align}
s_F = [U_F(a,a)x^2+(U_F(a,b)+U_F(b,a))xy +U_F(b,b)y^2]
\end{align}
where $(a, b)$ is a symplectic basis for $H_1(F;\Z)$ (i.e.~a basis with respect to which the intersection form is given by $\big(\begin{smallmatrix}
  0 & 1\\
  -1 & 0
\end{smallmatrix}\big)$).
The $\SL_2(\Z)$-equivalence class $s_F$ is a well-defined invariant of $F$, and by minor abuse of notation, we will also refer to it as the ``Seifert form of $F$''.

\begin{theorem}\label{thmintro:main3d} Let $s_1$ and $s_2$ be $\SL_2(\Z)$-equivalence classes of IBQFs.
The pair $s_1$ and $s_2$ arises as the pair of Seifert forms of a disjoint genus one pair of Seifert surfaces if and only if there exist $a,c\in\Z$ such that $s_1$ and $s_2$ have discriminant $1-4ac$ and
$[ax^2+xy+cy^2]^{2}*s_1=s_2$.
\end{theorem}

Before we explain a purely algebraic formulation of Theorem~\ref{thmintro:main3d} and how Gauss composition arises, we describe consequences concerning non-isotopy of pushed-in Seifert surfaces. For this we use the if-statement of Theorem~\ref{thmintro:main3d}.

Arguably the simplest classical integer-valued knot invariant is the determinant, denoted by $D_K$ for a knot $K$. We choose the sign convention for which $D_K \equiv 1 \mod{4}$ for all knots $K$. 
With this convention, if a genus one Seifert surface for $K$ has Seifert form $[ax^2+bxy+cy^2]$, then $D_K$ equals the {discriminant} of $ax^2+bxy+cy^2$.
For a fixed non-zero integer $D$,
the \emph{oriented class group} $\mathcal{G}^+_D$ is the set of the $\SL_2(\Z)$-equivalence classes of \emph{primitive} IBQFs---those with coprime coefficients---with discriminant $D$, which form an abelian group under Gauss composition. Note that if $D$ is square-free, every IBQF with discriminant $D$ is primitive.
For $D \equiv 1 \mod{4}$, the neutral element is
$1=[x^2+xy+\frac{-D+1}{4}y^2]$.

\begin{theorem}\label{thmintro:main4d}
 Fix an integer $D \equiv 1 \mod{4}$ and suppose that there exist $a,c\in\Z$ with $1-4ac=D$ such that $[ax^2+xy+cy^2]^2\neq 1$.
 Then there exists a knot $K$ with determinant $D$ that has two disjoint genus 1 Seifert surfaces that are not ambiently isotopic in $\B^4$ after being pushed into the $4$-ball.

 Suppose that in addition there exist $a',c'\in\Z$ with $1-4a'c'=D$ such that 
\[[ax^2+xy+cy^2]^2\neq[a'x^2+xy+c'y^2]^2\neq 1.\]
 Then one of the above Seifert surfaces can be chosen such
that its Seifert form is any given class of a primitive IBQF of discriminant $D$.
\end{theorem}

Theorem~\ref{thmintro:main4d} should be seen as a special case of the robust construction given in this article.
In particular, we give an exact characterization for when one can distinguish disjoint pairs of genus one Seifert surfaces by the intersection form of the double branched cover.
Also, the primitivity assumption in the theorem is removed below using more language concerning IBQFs which we avoid here for sake of exposition (cf.~Theorem~\ref{thm:main4d}).

It turns out the condition $[ax^2+xy+cy^2]^2\neq 1$ is often easy to check in the oriented class group. In fact, we have the following explicit result for negative determinants. (We note that given our sign convention, negative determinants correspond to the case of a definite symmetrized Seifert form or, equivalently, knots with non-zero signature.)

\begin{cor} \label{corintro}
 Fix a negative integer $D \equiv 1 \mod{4}$. There exist a knot $K$ with determinant $D$ and two Seifert surfaces with boundary $K$ that are not ambiently isotopic in $\B^4$ after being pushed into the $4$-ball if
 \begin{align*}
     \frac{-D+1}{4}\notin\left\{p^k\mid p\text{ prime and }k\in\{0,1,2\}\right\}.
 \end{align*}
\end{cor}

\begin{example}\label{exp:HMKPS1}
The largest negative integer $D$ with $D \equiv 1 \mod{4}$ such that $\frac{-D+1}{4}$ is not $1$, a prime, or a square of a prime is $-23$. The example in \cite{HMKPS22} corresponds exactly to this case.
Up to taking inverses there is exactly one non-trivial element in $\mathcal{G}^+_{-23}$ that is the square of an element of the form $[ax^2+xy+cy^2]$:
\[ [2x^2\pm xy +3y^2]^2=[-(2x^2\mp xy +3y^2)]^2=[2x^2\mp xy +3y^2] \neq 1=[x^2+xy+6y^2].\] Hence, by Theorem~\ref{thmintro:main4d}, there must exist pairs of Seifert surfaces for a knot with determinant $-23$ that are not isotopic in $\B^4$ after pushing in, as is claimed in Corollary~\ref{corintro}. See Example~\ref{ex:-23} for more details and the explicit connection to the example from~\cite{HMKPS22}.
\end{example}

Feh\'er \cite{Feher} has exhibited more examples of disjoint genus one pairs of Seifert surfaces.
We put these examples in  the number-theoretic framework of this article (cf.~Example~\ref{exp:Feher}).

Our framework readily allows treating positive discriminants as well.

\begin{example}
The smallest positive discriminant $D\equiv 1 \mod 4$ for which Theorem~\ref{thmintro:main4d} provides an example is $D = 25$.
There exists an element of $\mathcal{G}_{25}^+$ which is a non-trivial square of an element of the form $[ax^2+xy+cy^2]$, namely $[-x^2+5xy]$ (cf.~\S\ref{sec:squaredisc}).
Hence, Theorem~\ref{thmintro:main4d} applies. 
In fact, the above can be extended to all odd square discriminants $D\geq 25$; see Corollary~\ref{cor:nontrivclass-square}.
\end{example}

The topological idea behind our proof is that the data of a knot with two disjoint genus 1 Seifert surfaces is the same thing as that of a  closed genus 2 surface  $\Sigma$ (the union of the two Seifert surfaces) along with a non-contractible separating closed curve in $\Sigma$ (the knot).  Moreover, the Seifert pairings on the two disjoint Seifert surfaces can be computed algebraically as the restrictions of the Seifert form on $H_1(\Sigma; \Z) \cong \Z^4$ to two complementary two-dimensional subspaces.

As a result, the key algebraic object that governs which Seifert forms arise from disjoint genus one pairs is a particular bilinear form on $\Z^4$, which is isomorphic to the Seifert form of an arbitrary closed genus 2 surface in $S^3$.
Below we describe it in terms of the quadratic form it defines and its antisymmetrization $\theta$ and all that follows are results about this bilinear form that are independent of the knot theory above.

\subsection{Gauss composition via planes in $\Z^4$}\label{sec:numbertheoretic results}

We call a rank two oriented summand $L\subset \Z^4$ a \emph{plane}. Following~\cite{AEW-2in4}, we parametrize planes $L$ by two so-called Klein vectors $a_1(L),a_2(L)$. In our setting, the latter can be thought of as a pair of IBQFs of equal discriminant.
The correspondence between these pairs of IBQFs and planes is given by explicit formulas both ways (cf.~\S\ref{sec:Kleinmap}).
For an equivalence class $[q]$ we denote by $\overline{[q]} = [q(-x,y)]$ the orientation-reversed class.
The following is the main number-theoretic result in this article. To state it, we consider $\Z^4$ equipped with the quadratic form
\begin{align*}
(x_1,x_2,x_3,x_4)\mapsto x_1x_2+x_3x_4
\end{align*} and
for a plane $L$ with oriented basis $(v_1,v_2)$, we write $q_L$ for the IBQF obtained by evaluating this quadratic form at $xv_1+yv_2$. The $\SL_2(\Z)$-equivalence class $[q_L]$ is independent of the choice of oriented basis, justifying the notation.

\begin{theorem}\label{thm:gausscompositionviaKleinvectors}
For any plane $L$ with $\disc(q_L) \neq 0$, we have
\begin{align}\label{eq:relationforplane}
       [q_{L
        }] = \overline{[a_1(L)]}*[a_2(L)].
    \end{align}
\end{theorem}

One may see the above theorem as a geometric point of view on Gauss composition. Indeed, here is a concrete realization of Gauss composition:
Given a pair of, say, primitive IBQFs $q_1=a_1x^2+b_1xy+c_1y^2$ and $q_2=a_2x^2+b_2xy+c_2y^2$ of equal discriminant, consider 
\begin{equation}\label{eq:GaussMagic}
L=\set{M\in \quat(\Z): \SmMat{-b_1}{-2a_1}{2c_1}{b_1}M=M\SmMat{b_2}{-2a_2}{2c_2}{-b_2}}
\end{equation}
where $\quat$ denotes the algebra of two by two matrices. 
Since $q_1$ and $q_2$ have the same discriminant, $L$ is a free rank two module over $\Z$, which is given as the solution of the four concrete linear equations in \eqref{eq:GaussMagic}. 
One finds a basis $v_1,v_2$ of $L$. 
Up to switching these basis vectors, the composition of $[q_1]$ and $[q_2]$ is $[q_1]*[q_2]=[q]$ where $q(x,y)=\det(xv_1+yv_2)$; see \S\ref{sec:Kleinrelations} for more details.

To summarize, composing two classes amounts to computing a basis of an explicit plane.
We remark that classical forms of Gauss composition usually reduce to solving congruence equations modulo the first coefficients of $q_1,q_2$ (which can be arranged to be coprime).
Upon close inspection, a similar statement is true for the above basis search after elementary manipulations.

Theorem~\ref{thm:gausscompositionviaKleinvectors} provides a geometric perspective on Gauss composition, an idea related to \cite[Artikel 236]{gauss}.
It also yields a geometric recipe that, given two IBQFs, finds a Bhargava cube \cite{bhargava_composition} that gives rise to their composition (cf.~\S\ref{sec:Bhargava}).

\begin{remark}
Let us briefly remark on the historical context surrounding Theorem~\ref{thm:gausscompositionviaKleinvectors}.
    In \cite{AEW-2in4}, planes and their associated integral quadratic forms were studied from a viewpoint of equidistribution (typical statistical independence of $[q_L]$ and $[q_{L^{\perp}}]$, where $L^\perp$ is the orthogonal complement with respect to $x_1x_2+x_3x_4$). Here, the ambient quadratic form can be any norm form on a quaternion algebra over $\Q$.
In the process of the proof, the authors found a version of the relations in Theorem~\ref{thm:gausscompositionviaKleinvectors}
(see \cite[Thm.~8.1\& App.~B]{AEW-2in4}) though the exact formulation was irrelevant to the equidistribution purpose.
Chinta and Pereira Júnior \cite{ChintaPereira} later proved a slightly weaker form of these relations for the quadratic form $x_1^2+x_2^2+x_3^2+x_4^2$.
The proof given here of the relations in \eqref{eq:relationforplane} could be carried out via the adelic machinery established in \cite{AEW-2in4}.
\end{remark}

\subsection{Symplectic planes and applications to topology}\label{sec:Intro:sympplanes}
In addition to the quadratic form, we equip $\Z^4$ with the integral symplectic form $\theta(\cdot,\cdot)$ determined by
\begin{align*}
& \theta(e_1,e_2)=-\theta(e_2,e_1)=\theta(e_3,e_4)=-\theta(e_4,e_3)=1,\\
& \theta(e_i,e_j)=0 \text{ for } \{i, j\}\notin\left\{\{1,2\},\{3,4\}\right\}.
\end{align*}
where $(e_1,e_2,e_3,e_4)$ is the standard basis. 
A plane $L$ is called \emph{symplectic} if it has an oriented basis $(v_1,v_2)$ with $\theta(v_1,v_2)=1$.
Its symplectic complement is defined as $L^\pperp=\set{x\in \Z^4: \theta(x,y)=0 \text{ for all } y\in L}$ (with the orientation that makes it symplectic).
As we will see in \S\ref{sec:knottheorysetup}, in order to establish the existence of a disjoint genus one pair of Seifert surfaces for a knot with prescribed determinant, it suffices to answer the following question.
For which discriminants is there a symplectic plane $L$ with 
\begin{align}\label{eq:DiffFormsIntro}
[q_L]\neq [q_{L^\pperp}],
\overline{[q_{L^\pperp}]}\,?
\end{align}

We prove in particular the following.

\begin{cor}\label{cor:symplecticplanes}
Given an integer $D \equiv 1 \mod 4$, there exists a symplectic plane $L \subset \Z^4$ such that $q_L$ has discriminant $D$ and $[q_L] \ne [q_{L^\pperp}], \overline{[q_{L^\pperp}]}$ if and only if there exist $a, c \in \Z$ with $1-4ac = D$ and $[ax^2 + xy + cy^2]^2 \neq 1$.
\end{cor}

Corollary~\ref{cor:symplecticplanes} and variants thereof are the main input needed for the theorems in \S\ref{sec:topological results} as we will see in \S\ref{sec:knottheorysetup}.
For illustration, we briefly discuss how Corollary~\ref{cor:symplecticplanes} relates to Theorem~\ref{thm:gausscompositionviaKleinvectors}; a full proof is given in \S\ref{sec:qf on complements}.

\begin{proof}[Sketch of proof of Corollary~\ref{cor:symplecticplanes}]
For a plane $L$, a simple observation (cf.~\S\ref{sec:Kleinsymplectic}) shows that $a_2(L)$ is an IBQF of the form $ax^2+xy+cy^2$ if and only if $L$ is symplectic. 
Moreover, in this case $a_1(L^\pperp)=-a_1(L)$ and $a_2(L^\pperp) = -\overline{a_2(L)}$ as can be shown by elementary calculations (cf.~Proposition~\ref{prop:KleinForSymp}; here $\overline{\cdot}$ denotes the orientation reversal as earlier in this introduction).
Altogether, applying Theorem~\ref{thm:gausscompositionviaKleinvectors} twice (for $L$ and $L^\pperp$) we get
\begin{align}\label{eq:a2translate}
       [a_2(L)]^{-2} * [q_L]=\overline{[q_{L^\pperp}]}.
\end{align}
Choosing $[a_1(L)] = [a_2(L)]$ one obtains a symplectic plane as in \eqref{eq:DiffFormsIntro} whenever there exists an IBQF of the form $ax^2+xy+cy^2$ of the same discriminant with $[ax^2+xy+cy^2]^2 \neq 1$.
\end{proof} 

\subsection*{Organization of the paper}
This paper is organized as follows. 

In \S\ref{sec:notation}, we discuss some notation concerning integral binary quadratic forms valid throughout the article. In particular, we describe Gauss composition via concordant forms (Dirichlet's formulation), and we identify binary forms with traceless matrices.

In \S\ref{sec:Kleinmap}, we reconstruct the Klein map from \cite[\S2]{AEW-2in4} for planes in the algebra of $2\times2$ matrices.
We find the restriction on values of the Klein map at symplectic planes in \S\ref{sec:Kleinsymplectic}.

In \S\ref{sec:Kleinrelations}, we establish relations between the quadratic forms of planes and their associated Klein vectors.
In particular, this permits a comparison between the quadratic forms of planes and their orthogonal (or symplectic) complements.
In \S\ref{sec:Bhargava}, we relate our results to Bhargava's approach \cite{bhargava_composition} to Gauss composition.

In \S\ref{sec:knottheorysetup}, we recall various notions from knot theory and prove Theorems~\ref{thmintro:main3d} and~\ref{thmintro:main4d}. We also discuss our results in the context of S-equivalence and the previous results of the third author \cite{Miller_22} (in \S\ref{ss:Sequiv}) as well as the realization problem for a fixed knot (in \S\ref{sec:realizationfixed}).

In \S\ref{sec:examples}, we provide various examples to illustrate our main theorems for varying discriminants. 
In particular, we prove Corollary~\ref{corintro} in \S\ref{sec:negdisc}.

In \S\ref{sec:openproblems}, we discuss a variety of open problems relating to the results of the present article.

\subsection*{AI declaration}
AI tools (Opus 5.5 and GPT 6) were used to correct grammar and spot minor issues (for instance in Example~\ref{exp:Feher}).
The first version of this article appeared in 2023 and did not use any AI tools.

\subsection*{Acknowledgments}
We thank Ken Baker, Gautam Chinta, Manfred Einsiedler, Zsombor Feh\'er, Hari Iyer, Kyle Hayden, and Filippos Sytilidis for helpful discussions. We in particular thank Zsombor for his very careful reading of an earlier version and for detailed feedback.
We also wish to thank the referee for many useful comments improving the exposition.

\section{Preliminaries}\label{sec:notation}
We collect some notation and standard facts, pertaining mostly to integral binary quadratic forms (IBQFs).
For an IBQF $q(x,y) = ax^2 +bxy+cy^2$ we recall that the discriminant is $\disc(q) = b^2-4ac= -4\det(M)$ when $M$ is the symmetric matrix representing $q$.
An integer $D$ is correspondingly called a discriminant if $D \equiv 0,1 \mod 4$.
In the topological part of this article, we have $D \equiv 1 \mod 4$.
The content of $q$, denoted by $\content(q)\in \Z_{\geq 1}$, is defined to be the greatest common divisor of the coefficients of $q$.
We say that $q$ is primitive if its content is $1$.

The group $\GL_2(\Z)$ acts by composition on the set of IBQFs preserving the discriminant as well as the content.
The $\SL_2(\Z)$-equivalence class of an IBQF $q$ is denoted by $[q]$.
For $D\neq 0$ we denote by $\mathcal{Q}_D^+$ the set of $\SL_2(\Z)$-equivalence classes of IBQFs of discriminant $D$ and by $\mathcal{G}_D^+$ the subset of equivalence classes of primitive forms.
By reduction theory, the sets $\mathcal{Q}_D^+$ and $\mathcal{G}_D^+$ are finite; see for example \cite[Ch.~9,~Thm.~1.1]{cassels}.

\subsection{The composition law for integral binary quadratic forms}\label{sec:Dirichlet composition}
We recall the composition law for primitive forms.
In general, there is no unique composition of classes of non-primitive quadratic forms.  However, in the case where two classes $[q_1]$, $[q_2]$ have coprime content, it is known  that the composition $[q_1]\ast [q_2]$ coming from the Dirichlet law is well-defined.  This was written down first in \cite{pall-nonprimitive}, and generalized to quadratic forms over integral domains in \cite{butts-estes}.  For a more modern perspective, see Penner \cite{penner_geomofgauss}. 

We will call two IBQFs $q_1 = a_1 x^2 + b_1 xy + c_1 y^2$ and $q_2 =  a_2 x^2 + b_2 xy + c_2 y^2$ of equal non-zero discriminant $D$ \emph{concordant} if the middle coefficients agree, i.e.~$b_1 =b_2$, $a_1,a_2$ are non-zero, and $a_1a_2$ divides $\frac{D-b_1^2}{4}$.

\begin{prop}[{\cite[Thm.~2]{pall-nonprimitive}}]\label{prop:dirichlet law}
Let $s_1$, $s_2$ be classes of IBQFs of discriminant $D$, with content $m_1$ and $m_2$ respectively, such that $\gcd(m_1, m_2) = 1$.
Then one can find a (not necessarily unique) pair of concordant forms of the form $q_1 = a_1 x^2 + b xy + a_2 c y^2 \in s_1$ and $q_2 =  a_2 x^2 + b xy + a_1 c y^2\in s_2$ respectively.
The class of
\[
q_1 q_2 = a_1 a_2 x^2 + b xy + cy^2
\]
is independent of the choices of $q_1$ and $q_2$ and has content $m_1m_2$.
Moreover, this gives rise to a well-defined composition law (known as \emph{Dirichlet composition})
\begin{align*}
\{s_1 \in \mathcal{Q}_D^+&: \mathrm{content}(s_1) = m_1\}
\times
\{s_2 \in \mathcal{Q}_D^+: \mathrm{content}(s_2) = m_2\} \\
&\longrightarrow
\{s \in \mathcal{Q}_D^+: \mathrm{content}(s) = m_1m_2\},
\end{align*}
whenever $\gcd(m_1, m_2) = 1$.
\end{prop}

\begin{prop}\label{prop:actionwelldef}
The above composition law is associative when defined.
\end{prop}
\begin{proof}
This is proved in \cite[Thm.~3.5]{butts-estes} in the general case where $\mathbb{Z}$ is replaced with a general integral domain.
\end{proof}

In particular, we obtain an associative composition law $\mathcal{G}_D^+\times \mathcal{G}_D^+ \to \mathcal{G}_D^+$.
In fact, this composition law turns $\mathcal{G}_D^+$ into a (finite) abelian group.
Indeed, one may verify that the identity element is given by 
\begin{align*}
1=1_D=\begin{cases}
[x^2-\frac{D}{4}y^2] &\text{if } D \equiv 0 \mod 4\\
[x^2+xy-\frac{D-1}{4}y^2] &\text{if } D \equiv 1 \mod 4.
\end{cases}
\end{align*}
(Here, note that any discriminant $D$ is congruent to either $0$ or $1$ mod $4$.)
Moreover, setting $\bar{q}(x,y) = q(-x,y)$ --- this represents a choice of orientation reversal --- one may check that the class $[\bar{q}]$ satisfies $[\bar{q}] \ast [q] = 1_D$ for any $[q] \in \mathcal{G}_D^+$.
Equivalently, $[q]^{-1} = [\bar{q}]$.
We sometimes write $\overline{[q]} = [\overline{q}]$; note that $\overline{[q]}$ indeed only depends on the class $[q]$ of $q$.

Classically, Gauss discovered the composition law on the set of classes of IBQFs of a fixed discriminant; hence it is often called the Gauss composition law.
The above approach through concordant forms is due to Dirichlet.
Note that Gauss composition does not yield a composition rule on $\GL_2(\Z)$-equivalence classes of IBQFs with a given discriminant.

\begin{remark}
While the oriented class group $\mathcal{G}_D^+$ features prominently in the literature, the most classical object is the class group $\mathcal{G}_D$; the group of $\PGL_2(\Z)$-equivalence classes of primitive IBQFs of discriminant $D$ where $\PGL_2(\Z)$ acts through $g.q = \frac{1}{\det(g)}q \circ g^t$.
In particular, we have an obvious surjective homomorphism $\mathcal{G}_D^+ \to \mathcal{G}_D$.

Alternatively, $\mathcal{G}_D$ is the quotient of $\mathcal{G}_D^+$ by the $\Z/2\Z$-action given by $[q] \mapsto [-\overline{q}]$.
In particular, the map $\mathcal{G}_D^+ \to \mathcal{G}_D$ is at most $2$-to-$1$ and we have the following.
\begin{itemize}
    \item When $D<0$, the $\SL_2(\Z)$-action preserves the signature and hence the map $\mathcal{G}_D^+ \to \mathcal{G}_D$ is $2$-to-$1$.
    \item When $D>0$ and a solution to the negative Pell equation exists, $\mathcal{G}_D^+ \to \mathcal{G}_D$ is an isomorphism. 
    Recall that for $D \equiv 1 \mod 4$ (resp.~$D\equiv 0 \mod 4$) the negative Pell equation is $x^2+xy-\frac{D-1}{4}y^2 = -1$ (resp.~$x^2-\frac{D}{4}y^2 = -1$).
    If no solution exists, it is again $2$-to-$1$.
    In terms of hyperbolic geometry, a solution to the negative Pell equation guarantees that the geodesic associated to any $[q] \in \mathcal{G}_D^+$ (obtained through connecting the zeros of $q$ on the boundary $\mathbb{P}^1(\R) \simeq \partial \mathbb{H}$) is equivalent via $\SL_2(\Z)$ to its time reversal. 
\end{itemize}
\end{remark}


The following corollary of Proposition~\ref{prop:actionwelldef} relates the oriented class groups $\mathcal{G}_D^+$ and $\mathcal{G}_{D/m^2}^+$ for a square divisor $m^2 \mid D$.

\begin{cor}\label{cor:modconductor}
Let $m$ be an integer with $m^2 \mid D$ such that $\frac{D}{m^2}$ is a discriminant.
The group $\mathcal{G}_D^+$ of equivalence classes of primitive IBQFs of discriminant $D$ acts on the set of equivalence classes of IBQFs of discriminant $D$ and content $m$ by composition.
Moreover, there is a surjective homomorphism $\mathcal{G}_D^+ \to \mathcal{G}_{D/m^2}^+, t \mapsto \tilde{t}$
with the following compatibility rule:
for any $t \in \mathcal{G}_D^+$ and for any $s \in \mathcal{Q}_D^+$ of content $m$ we have
\begin{align}\label{eq:reductionComp}
    \tfrac{1}{m}(t * s) = \tilde{t} * (\tfrac{1}{m}s).
\end{align}
\end{cor}

\begin{proof}
The first part is immediate from Proposition~\ref{prop:actionwelldef}.
For the second part, the map associating to $t \in \mathcal{G}_{D}^+$ the class $\tilde{t} = \frac{1}{m} (t \ast (m\,1_{D/m^2})) \in \mathcal{G}_{D/m^2}^+$ satisfies the requirements -- see \cite[p.~2706]{penner_geomofgauss}.
\end{proof}

The following lemma describes how the composition law interacts with the maps $[q] \mapsto [\overline{q}]$ and $[q] \mapsto [-q]$.

\begin{lem}\label{lem:flipthingsaround}
Suppose that $q_1,q_2,q_3$ are IBQFs of equal discriminant and that $q_2,q_3$ have coprime content.
If $[q_1] = [q_2]*[q_3]$, then 
\begin{enumerate}[a)]
    \item $[\overline{q_1}] = [\overline{q_2}] * [\overline{q_3}]$,
    \item $[-q_1] = [\overline{q_2}] * [-q_3] = [-q_2] * [\overline{q_3}]$, and
    \item $[-\overline{q_1}] = [q_2]*[-\overline{q_3}] = [-\overline{q_2}]*[q_3]$.
\end{enumerate}
\end{lem}

\begin{proof}
By Proposition~\ref{prop:dirichlet law} we may assume that $q_2,q_3$ are concordant of the form
$q_2=a_2x^2+bxy+a_3cy^2$ and $q_3=a_3x^2+bxy+a_2cy^2$.
Moreover, we may assume $q_1=a_2a_3 x^2+bxy+cy^2$.

For a), the forms $\overline{q_2}=a_2x^2-bxy+a_3cy^2$ and
$\overline{q_3}=a_3x^2-bxy+a_2cy^2$ are again concordant, and the composition of their classes is the class of
$a_2a_3x^2-bxy+cy^2=\overline{q_1}$ as claimed.
The claims in b) and c) follow by the same logic.
\end{proof}

For later purposes, when $D \equiv 1\mod 4$ we define $\mathcal{S}^+_D$ to be the subgroup of $\mathcal{G}^+_D$ generated by the set
\begin{align}\label{eq:generatingset}
\{[ax^2+xy+cy^2]^2: a,c \in \Z \text{ with }1-4ac=D\}.
\end{align}
Note that \eqref{eq:generatingset} (and hence also $\mathcal{S}^+_D$) is invariant under $[q] \mapsto [\bar{q}]$ as $[ax^2-xy+cy^2] = [cx^2+xy+ay^2]$ for any $a,c$.

\subsection{Relation to traceless matrices}\label{sec:qf vs traceless}
Binary forms may be identified with traceless matrices as follows.
Denote by $\quat$ the space of $2\times 2$ matrices and by $\quatzero$ the subspace of traceless matrices in $\quat$.
Define the so-called Gross lattice
$$
\GrossLat=\set{A\colon A\in \SmMat{\Z}{2\Z}{2\Z}{\Z}\cap \quatzero}\subset \quatzero(\Z).
$$
There is a linear bijection between the set of IBQFs and $\Lambda$.
Indeed, we identify a binary quadratic form $q=ax^2+bxy+cy^2$ with the traceless matrix $A(q)=\SmMat{b}{-2a}{2c}{-b}$. 
We denote the inverse of this identification by $A\in \quatzero\mapsto q_A$. 
 The identification is equivariant with respect to the $\SL_2(\Z)$-actions where one acts by conjugation on $\quatzero(\Z)$, that is, $A(g.q)=gA(q)g^{-1}$. 
 Note that $A(\cdot)$ restricts to an equivariant map from $\{q:\disc(q)=D\}$ to traceless matrices in $\Lambda$ of determinant $-D$. 
 Moreover, $A(q)$ is primitive in the Gross lattice $\Lambda$ if and only if $q$ is primitive.

\section{The Klein map}\label{sec:2in4Matrices}\label{sec:Kleinmap}

In this section, we set up the correspondence between planes in $\Z^4$ and Klein vectors as in \cite{AEW-2in4}. This was already alluded to in the introduction.
In \S\ref{sec:Kleinsymplectic}, we specialize the construction to symplectic planes.

\subsection{Setup and Notation}\label{subsec:Setup}

We may identify $\R^4$ with the space $\quat(\R)$ and $\bZ^4$ with $\quat(\Z)$.
The usual (standard) involution on $\quat$ is the adjugate operation, which we denote by $\overline{\big(\begin{smallmatrix}
  a & b\\
  c & d
\end{smallmatrix}\big)}=\big(\begin{smallmatrix}
  d & -b\\
  - c & a
\end{smallmatrix}\big)$. With respect to the basis $\Bcal=\pa{\big(\begin{smallmatrix}
  1 & 0\\
  0 & 0
\end{smallmatrix}\big),\big(\begin{smallmatrix}
  0 & 0\\
  0 & 1
\end{smallmatrix}\big),\big(\begin{smallmatrix}
  0 & 0\\
  -1 & 0
\end{smallmatrix}\big),\big(\begin{smallmatrix}
  0 & 1\\
  0 & 0
\end{smallmatrix}\big)}$ the bilinear form
$$
M(x,y)=\trace\pa{x\big(\begin{smallmatrix}
  1 & 0\\
  0 & 0
\end{smallmatrix}\big)\overline y}
$$
is given by the matrix 
\begin{equation}\label{eq:BilFormU}
U=\begin{pmatrix}
0 & 1 & 0 & 0\\
0 & 0 & 0 & 0\\
0 & 0 & 0 & 1\\
0 & 0 & 0 & 0
\end{pmatrix}.
\end{equation}
Since
$
M(y,x)=\trace\pa{y\big(\begin{smallmatrix}
  1 & 0\\
  0 & 0
\end{smallmatrix}\big)\overline x}=\trace\pa{\overline{y\big(\begin{smallmatrix}
  1 & 0\\
  0 & 0
\end{smallmatrix}\big)\overline x}}=\trace\pa{x\big(\begin{smallmatrix}
  0 & 0\\
  0 & 1
\end{smallmatrix}\big)\overline y}
$
the symmetric form $Q=M+M^t$ and the symplectic form $\theta=M-M^t$ are given by
\begin{equation}
\label{eq:forms}
Q(x,y)= \trace\pa{x\overline y},\quad \theta(x,y)= \trace\pa{x\Notj\overline y},
\end{equation}
where $\Notj$ denotes the matrix $\big(\begin{smallmatrix}
  1 & 0\\
  0 & -1
\end{smallmatrix}\big)$. As customary for quadratic forms, we set $Q(x)=Q(x,x)$ and note that $Q(x)=2\det x$.

\subsection{Planes and the Klein map}

Recall from \S\ref{sec:numbertheoretic results} that we call a rank two oriented summand $L \subset \Z^4$ a plane. 
Here, an orientation amounts to a choice of basis of $L$ considered up to the equivalence relation of applying integral basis changes with determinant 1.
In the following, any basis chosen for a plane will be taken to match the orientation.
Via the identification $\Z^4 \simeq \quat(\Z)$ from above, any plane will also be considered an oriented summand of $\quat(\Z)$.
Recall from the introduction that a plane $L$ defines an equivalence class $[q_L]$, where $q_L$ is defined via a choice of basis $(v_1,v_2)$ of $L$
as follows
\begin{multline}
\label{eq:QuadFormCoeff}
q_L(x,y):=\tfrac{1}{2}Q(xv_1+yv_2)=\tfrac{1}{2}\pa{Q(v_1)x^2+2Q(v_1,v_2)xy+Q(v_2)y^2}\\
= \det(v_1)x^2+\trace(v_1\overline{v_2}) xy+\det(v_2)y^2\in \Z[x,y].    
\end{multline}

Let $\Lcal$ denote the set of planes.
We set $\Lcal_ {\neq 0} = \{L \in \Lcal: \disc(q_L) \neq 0\}$ and $\Lcal_D=\set{L\in\Lcal:\disc(q_L)=D}$ for every non-zero discriminant $D$.

\begin{lem}
For any discriminant $D \equiv 0,1 \mod 4$, the set $\Lcal_D$ is non-empty.
\end{lem}

\begin{proof}
For $D \equiv 0 \mod 4$, $\Lcal_D$ contains the plane spanned by $\SmMat{1}{0}{0}{1}$ and $\SmMat{0}{\frac{D}{4}}{1}{0}$.
For $D \equiv 1 \mod 4$, $\Lcal_D$ contains the plane spanned by $\SmMat{1}{0}{0}{1}$ and $\SmMat{1}{\frac{D-1}{4}}{1}{0}$.
\end{proof}

 Similarly to \cite{AEW-2in4} we would like to study $\Lcal_D$ via two associated vectors in $\quatzero(\Z)$.
 The reader can compare this discussion to \cite[\S 2]{AEW-2in4} but notice that here we consider planes (or sublattices of $\Z^4$) \emph{with} orientation (cf.~\cite[\S 7]{AEW-2in4}).
 
Note that $\SL_2(\Z) \times \SL_2(\Z)$ acts on $\quat(\Z)$ by $(g_1,g_2).x = g_1xg_2^{-1}$ which clearly preserves the determinant.
In particular, this action induces an action of the group $\SL_2(\Z) \times \SL_2(\Z)$ on $\Lcal_D$ for every $D$.
 
Given a plane $L$ we write $L_{\mathrm{op}}$ for the same plane with reversed orientation.
Note $[q_{L_\mathrm{op}}] = [\overline{q_L}]$.
For a plane $L\in \Lcal$ we choose a basis $(v_1,v_2)$ for $L$
and consider the traceless matrices
\begin{equation}\label{eq:DefOfKleinMap}
a_1(L)=2v_1\overline{v_2}-\trace\pa{v_1\overline{v_2}}I_2,\quad
a_2(L)=2\overline{v_2}v_1-\trace\pa{\overline{v_2}v_1}I_2.    
\end{equation}
We call $a_1(L)$ and $a_2(L)$ the Klein vectors (“vectors” in the three-dimensional vector space $\quatzero$) associated to $L$.
By definition, $a_1(L)$ and $a_2(L)$ are elements of the Gross lattice $\Lambda$.
We observe that the bilinear maps $(v_1,v_2) \mapsto 2v_1\overline{v_2}-\trace\pa{v_1\overline{v_2}}I_2$ and $(v_1,v_2) \mapsto 2\overline{v_2}v_1-\trace\pa{\overline{v_2}v_1}I_2$ are alternating.
In particular, the definition \eqref{eq:DefOfKleinMap} is independent of the choice of basis.

\begin{definition}[pair-primitivity] \label{def:pairprimitive}
We call two matrices $a_1,a_2\in \GrossLat$ \emph{pair-primitive} if $(a_1,a_2)$ is a primitive element of $\GrossLat^2\subset \quatzero(\bQ)^2$. That is, there is no prime $p$ such that $\tfrac1p a_i\in \GrossLat$ for both $i=1,2$.
\end{definition}
We set
\begin{align*}
\Kcal_{\neq 0} &= \set{(a_1,a_2)\in \GrossLat^2\colon 0\neq\det(a_1) = \det(a_2),\, a_1,a_2 \text{ pair-primitive}},\\
\Kcal_D &= \{(a_1,a_2)\in \Kcal_{\neq 0}\colon \det(a_1) = -D\}.
\end{align*}
Note that $\SL_2(\Z) \times \SL_2(\Z)$ acts on these sets through individual conjugation, that is, $(g_1,g_2).(a_1,a_2) = (g_1a_1g_1^{-1},g_2a_2g_2^{-1})$.
The aim of this subsection is to show that $\Kcal_D$ parametrizes elements of $\Lcal_D$.

\begin{proposition}[The Klein map] \label{prop:KleinBasicProp}
Consider the map
$$
\Phi:\Lcal_{\neq 0}\to \Lambda \times \Lambda,\, L\mapsto (a_1(L),a_2(L)).
$$
The following holds:
\begin{enumerate}[(1)]
    \item \label{item:KleinIsArith} For $L\in\Lcal_D$ and for $i=1,2$ we have 
    $\det(a_i(L))=-\disc([q_L])=-D$. 
    \item \label{item:ResOrient} We have $\Phi(L_{\mathrm{op}})=(-a_1(L),-a_2(L)) = -\Phi(L)$.
    \item \label{item:PhiInv} 
    The map $\Phi$ induces a bijection $\Lcal_D \to \Kcal_D$ for any discriminant $D \neq 0$ and, in particular, a bijection $\Lcal_{\neq 0} \to \Kcal_{\neq 0}$.
    Moreover,  the inverse is given by 
    \begin{equation}\label{eq:nielK}
    \Psi\colon \Kcal_{\neq 0}\to \Lcal_{\neq 0},\quad (a_1,a_2)\mapsto L_{a_1,a_2}:=\set{x\in \quat(\Z):a_1x=xa_2},
    \end{equation}
    where $L_{a_1,a_2}$ is equipped with the orientation induced by the
non-zero alternating form\footnote{For a non-zero alternating form $\omega$ a basis $(v_1,v_2)$ is oriented if $\omega(v_1,v_2) >0$.}
$(x,y) \mapsto -\frac{\trace(a_1x\overline y)}{\det(a_1)}$.
        \item \label{item:KleinVecPerp} 
        For $L\in\Lcal_{\neq 0}$ with Klein vectors $a_1,a_2$ and $L'$ a plane orthogonal to $L$ we have either $L'=L_{-a_1,a_2}$ or $L' = L_{a_1,-a_2}$.
        In particular, $L' \in \Lcal_D$ if $L \in \Lcal_D$.
    \item\label{item:Phiequiv}
    The map $\Phi$ is $\SL_2(\Z)\times \SL_2(\Z)$-equivariant.
\end{enumerate}
\end{proposition}

Disregarding the orientation, the above proposition can be seen as a parametrization of (non-oriented) planes via $\mathcal{K}_{\neq 0}/\{\pm (1,1)\}$.
We also refer readers to \cite[\S2]{AEW-2in4} for an orientation-free version.
We regard the oriented version of the above proposition as more natural, particularly in the context of \cite{AEW-2in4} (indeed, $\Kcal_D$ is a set of integer points on an affine homogeneous variety).
The reader is nevertheless encouraged to ignore all orientation-related issues in a first reading of the proof.

\begin{proof}[Proof of Proposition~\ref{prop:KleinBasicProp}]
For Item~\ref{item:KleinIsArith} we calculate:  
\begin{align*}
a_1(L)\overline{a_1(L)}
    &=
    \pa{2v_1\overline{v_2}-\trace\pa{v_1\overline{v_2}}I_2}\pa{2v_2\overline{v_1}-\trace\pa{v_1\overline{v_2}}I_2}\\
    &=\pa{4\det(v_1) \det(v_2)I_2-2\trace\pa{v_1\overline{v_2}}\pa{v_1\overline{v_2}+\overline{v_1\overline{v_2}}} +\pa{\trace\pa{v_1\overline{v_2}}}^2I_2}\\
    &= \pa{Q(v_1)Q(v_2) -\pa{\trace\pa{v_1\overline{v_2}}}^2}I_2. 
\end{align*}
Hence, 
\[\det(a_1(L))=\frac{1}{2}Q\pa{a_1(L)}=\frac{1}{2}\trace{a_1(L)\overline{a_1(L)}}=Q(v_1)Q(v_2) -\pa{\trace\pa{v_1\overline{v_2}}}^2.\]
The same calculation for $Q(a_2(L))$ yields the same expression, so $Q(a_1(L))=Q(a_2(L))$. Moreover, by writing out $q_L(x,y)$ as in \eqref{eq:QuadFormCoeff} with the basis $(v_1,v_2)$ we see that 
\begin{align}\label{eq:detL_VS_det_ai}
-\disc q_L=4\det(v_1)\det(v_2)-Q(v_1,v_2)^2
&= Q(v_1)Q(v_2)-\pa{\trace\pa{v_1\overline{v_2}}}^2\nonumber\\
&=\det(a_i(L))
\end{align}
for $i=1,2$, as needed.

Item \ref{item:ResOrient} is immediate from the fact that the expressions in \eqref{eq:DefOfKleinMap} as bilinear maps of $(v_1,v_2)$ are alternating, a fact that we have already exploited to conclude that $a_1(L),a_2(L)$ are independent of the choice of basis (see the discussion after \eqref{eq:DefOfKleinMap}).

We start the proof of item \ref{item:PhiInv} by showing the containment $\Phi(\Lcal_D)\subset \Kcal_D$.
With item \ref{item:KleinIsArith} pair-primitivity is all we need to prove. 
This can be proven similarly to \cite[Lemma 2.4]{AEW-2in4}; we give a variant thereof. 
Since \eqref{eq:DefOfKleinMap} are bilinear and alternating expressions in $(v_1,v_2)$, the map $\Phi$ factors through $\bigwedge^2\quat$:
\begin{center}
\begin{tikzcd}[column sep=3em]
\Lcal \arrow[rd] \arrow[r] & \bigwedge^2\quat(\Z) \arrow[d, "\tilde{\Phi}"']\\
& \quatzero(\Z)\times \quatzero(\Z) 
\end{tikzcd}
\end{center}
Denote by $(b_1,b_2,b_3,b_4)$ the basis $\Bcal$ fixed in \S \ref{subsec:Setup} and consider the basis $(c_1=b_1-b_2,c_2=2b_3,c_3=2b_4)$ of $\GrossLat$. 
  Representing $\tilde{\Phi}$ with respect to the bases
  \begin{align*}
&(b_1\wedge b_2,b_1\wedge b_3,b_1\wedge b_4,b_2\wedge b_3,b_2\wedge b_4, b_3\wedge b_4) \text{ and }\\
&((c_1,0),(c_2,0),(c_3,0),(0,c_1),(0,c_2),(0,c_3))
  \end{align*}
we get the matrix
\begin{equation*}
\begin{pmatrix}
1   & 0 & 0 & 0 & 0 & -1 \\   
0   & 0 & 0 & -1 & 0 & 0 \\  
0   & 0 & -1 & 0 & 0 & 0 \\  
1   & 0 & 0 & 0 & 0 & 1 \\  
0   & -1 & 0 & 0 & 0 & 0 \\  
0   & 0 & 0 & 0 & -1 & 0     
\end{pmatrix}
\end{equation*}
which has determinant $-2$.
This shows that $\tilde{\Phi}$ yields a bijection between primitive vectors in $\bigwedge^2\quat(\Z)$ and primitive vectors in the index two sublattice
\begin{align*}
\Lambda' = \{(a_1,a_2) \in  \Lambda \times \Lambda: a_1+a_2 \in 2 \quat(\Z)\}
\end{align*}
of $\Lambda \times \Lambda$.
It remains to show that the set $\mathcal{K}'$ of primitive vectors $(a_1,a_2)$ in $\Lambda'$ such that $\det(a_1)= \det(a_2)\neq 0$ is equal to $\mathcal{K}_{\neq 0}$.
Clearly, if $(a_1,a_2) \in \mathcal{K}_{\neq 0}$ then the fact that $\det(a_1)=\det(a_2)$ implies that $a_1+a_2 \in 2 \quat(\Z)$ because $(a_1)_{11},(a_2)_{11}$ need to share the same parity.
Also, $(a_1,a_2)$ is primitive in $\Lambda'$ because it is primitive in the larger lattice $\Lambda \times \Lambda$. Thus, $\mathcal{K}_{\neq 0} \subset \mathcal{K}'$.
On the other hand, if $(a_1,a_2) \in \mathcal{K}'$ were a non-trivial multiple of $(a_1',a_2') \in \Lambda^2$ then $2(a_1',a_2') = (a_1,a_2)$.
But $a_1',a_2'$ have the same determinant and hence $a_1'+a_2' \in 2\quat(\Z)$ so that $(a_1',a_2') \in \Lambda'$.
This finishes the claims that $\mathcal{K'} = \mathcal{K}_{\neq 0}$ and $\Phi(\mathcal{L}_D) \subset \mathcal{K}_D$.

 For the reverse containment, we need first to discuss the inverse $\Psi$. 
 Let $(a_1,a_2)\in \Kcal_{\neq 0}$.
 Since $\det(a_1)=\det(a_2)$ and $a_1,a_2$ are traceless, they have the same characteristic polynomial $x^2-D$. Thus, there exists $g\in \quat^\times(\bQ)=\GL_2(\Q)$ such that $ga_2g^{-1}=a_1$ (both $a_1,a_2$ are conjugate to the companion matrix of $x^2-D$).
 In particular, $\Psi(a_1,a_2)=L_{a_1,a_2}\neq\{0\}$.
 Since the subspace $L_{a_1,a_2}(\Q) := L_{a_1,a_2}\otimes \Q$ satisfies $L_{a_1,a_2}(\Q) = L_{a_1,a_1}(\Q)g$ and the centralizer $L_{a_1,a_1}(\Q) = \Q(a_1)$ of $a_1$ has dimension $2$, $L_{a_1,a_2}$ has rank $2$.
 It remains to check that the orientation on $L_{a_1,a_2}$ is well-defined or equivalently that the alternating form $(x,y) \mapsto -\frac{\trace(a_1x\overline y)}{\det(a_1)}$ is non-zero.
This is immediate from $g,a_1g \in L_{a_1,a_2}(\Q)$ and $\frac{\trace(a_1g\overline{a_1 g})}{\det(a_1)} = 2\det(g) \neq 0$.
Thus, $\Psi$ is well-defined.

For later purposes, we also note that for any $x,y \in L_{a_1,a_2}\otimes \Q$ we have that $x\overline{y}$ commutes with $a_1$ and so $2x\overline y-\trace(x\overline y)I_2=c a_1$ for some $c\in\Q$; multiplying by $a_1$ and taking traces gives $c=-\frac{\trace(a_1x\overline y)}{\det(a_1)}$.
Arguing similarly with $\overline{y}x$ we obtain
\begin{align}\label{eq:alt form}
2x\overline y-\trace(x\overline y)I_2= \big(-\tfrac{\trace(a_1x\overline y)}{\det(a_1)}\big)a_1,\quad
2\overline yx-\trace(\overline yx)I_2=\big(-\tfrac{\trace(a_1x\overline y)}{\det(a_1)}\big)a_2.
\end{align}

 

We now show that $\Psi$ and $\Phi$ are inverse to each other.
Let $L\in\Lcal_D$ and let $(v_1,v_2)$ be an oriented basis of $L$. A direct calculation shows that $v_1,v_2\in L_{a_1(L),a_2(L)}$, hence $L=L_{a_1(L),a_2(L)}$ as non-oriented planes. Applying \eqref{eq:alt form} gives 
\begin{align*}
a_1(L)
=
\big(-\tfrac{\trace(a_1(L)v_1\overline{v_2})}{\det(a_1(L))}\big)a_1(L),
\end{align*}
so the factor in parentheses is equal to $1$. Thus, the orientation on $L_{a_1(L),a_2(L)}$ agrees with that of $L$, and therefore $\Psi\circ\Phi|_{\Lcal_D}=\Id_{\Lcal_D}$.

Conversely, let $(a_1,a_2)\in\Kcal_D$ and let $(v_1,v_2)$ be an oriented basis of $L_{a_1,a_2}$. By~\eqref{eq:alt form} we have $\Phi(L_{a_1,a_2})
=
(-\tfrac{\trace(a_1v_1\overline{v_2})}{\det(a_1)})(a_1,a_2)$ and the factor $-\tfrac{\trace(a_1v_1\overline{v_2})}{\det(a_1)}$ is positive by definition of the orientation.
Since both $\Phi(L_{a_1,a_2})$ and $(a_1,a_2)$ are primitive elements of $\GrossLat^2$, it must equal $1$. Hence $\Phi\circ\Psi|_{\Kcal_D}=\Id_{\Kcal_D}$, which completes the proof of Item~\ref{item:PhiInv}.


For Item~\ref{item:KleinVecPerp} let $L \in \Lcal_D$ and denote $a_i=a_i(L)$.
By Item~\ref{item:PhiInv} and the statement in Item~\ref{item:KleinVecPerp}, 
it suffices to show that the planes $\set{x\in \quat(\Z):a_1x=xa_2}$ and $\set{x\in \quat(\Z):-a_1x=xa_2}$ (each equipped with an arbitrary choice of orientation), are orthogonal.
Let $x,y\in \quat(\Q)$ with $a_1x=xa_2$ and $-a_1y=ya_2$.
Using that $a_i\overline{a_i}=\det(a_i)\cdot I_2$ for $i=1,2$, we have
\begin{multline*}
\det(a_1)\cdot Q(x,y)=\det(a_1)\cdot\trace(x\bar{y})=\trace(\bar{a_1}a_1x\bar{y})=\trace(a_1x\bar{y}\bar{a_1})=\\
=\trace(a_1x\overline{\pa{a_1y}})=\trace(xa_2\overline{\pa{-ya_2}})
=\trace(-xa_2\bar{a_2}\bar{y})=-\det(a_2)\cdot Q(x,y). 
\end{multline*}
By Item~\ref{item:KleinIsArith}, $\det(a_1)=\det(a_2)\neq 0$, thus $Q(x,y)=0$ as desired.

Item~\ref{item:Phiequiv} is a direct consequence of the definition of $a_1(\cdot)$ and $a_2(\cdot)$.
\end{proof}

Regarding Item~\ref{item:KleinVecPerp} in Proposition~\ref{prop:KleinBasicProp}, we note that, without orientation, the orthogonal complement $L^\perp$ of a plane $L$ (with non-zero discriminant) would be defined as 
\begin{align*}
L^\perp=\set{x\in \quat(\Z)\colon Q(x,\ell)=0,\,\forall\ell\in L}.
\end{align*}
As we would like to fix an orientation, we define in view of Item~\ref{item:KleinVecPerp} the orthogonal complement of an oriented plane as follows: for $(a_1,a_2) \in \mathcal{K}_{\neq 0}$
\begin{align}\label{eq:orthcompl-oriented}
(L_{a_1,a_2})^\perp = L_{-a_1,a_2}.
\end{align}

\begin{remark}
The choice of orientation on the orthogonal complement in \eqref{eq:orthcompl-oriented} is only one possible choice.
Alternatively, one may choose an orientation on $L^\perp$ such that its oriented basis together with the oriented basis of $L$ forms an oriented basis of $\R^4$.
These two choices only agree for planes $L$ with $\disc(q_L) >0$.
The choice here simplifies the exposition and is natural as it is induced by a linear automorphism on $\bigwedge^2 \Z^4$, as the above proof of Proposition~\ref{prop:KleinBasicProp} shows. (Compare this to the notion of the Hodge star operator, see e.g.~\cite[\S1]{Schmidt-subspace} or \cite[Ch.V.1]{Wells}). 
\end{remark}

\begin{remark}
Formulas similar to \eqref{eq:nielK} appear when exhibiting actions of the class group on e.g.~integer points of the sphere -- see for instance \cite{EllenbergMichelVenkatesh} and the references therein. 
\end{remark}

\subsection{Symplectic planes and their complements}\label{sec:Kleinsymplectic}
 We call a plane $L$ \emph{symplectic} if we can find a basis $(v_1,v_2)$ of $L$ with $\theta(v_1,v_2)=1$. Such a basis will also be called symplectic.  
For a symplectic plane $L$ we define the symplectic complement $L^\pperp$ as 
\[
L^\pperp=\set{x\in \quat(\Z)\colon \theta(x,\ell)=0,\,\forall\ell\in L},
\]
where the orientation is given by any symplectic basis.

\begin{lem}\label{lem:SympIf1}
    A plane $L\in\Lcal$ is symplectic if and only if $a_2(L)=\SmMat{1}{\alpha}{\gamma}{-1}$ for some $\alpha,\gamma\in 2\Z$.
    In particular, $\disc(q_L) \equiv 1 \mod 4$.
\end{lem}

\begin{proof}
Since $a_2(L)$ is in the Gross lattice $ \Lambda$, it has the form
  $\SmMat{\beta}{\alpha}{\gamma}{-\beta}$ for $\beta\in\bZ$ and $\alpha,\gamma\in2\Z$.
  Let $v_1,v_2$ be a basis of $L$.
By \eqref{eq:forms} we have
\begin{align*}
\theta(v_1,v_2)=\trace\pa{v_1\Notj\overline{v_2}}=\trace\pa{\overline{v_2}v_1\Notj}
=\trace\big(\tfrac{1}{2}a_2(L)\Notj\big)=\beta
\end{align*}
where the second to last equality follows as $\trace\pa{I_2\Notj}=0$.
Therefore, if $\beta =1$, the basis $(v_1,v_2)$ and $L$ are symplectic.
Conversely, if the basis is symplectic then $\beta =1$.
The lemma follows.
\end{proof}

Lemma~\ref{lem:SympIf1} together with the previous observations enable us to relate the Klein vectors and the associated quadratic forms for $L,\,L^\perp$ and $L^\pperp$:

\begin{proposition}\label{prop:KleinForSymp}
    Let $L\in \Lcal_D$ be  a symplectic plane. Then $L^\perp$ is symplectic. 
Furthermore, $L^\pperp \in \Lcal_D$ and $[q_{L^\pperp}] = [-\overline{q_{L^\perp}}]$.
If we write $a_2(L)=\SmMat{1}{\alpha}{\gamma}{-1}$, then 
    \begin{align}
        \Phi(L^\pperp)=(-a_1(L),\SmMat{1}{-\alpha}{-\gamma}{-1}).
    \end{align}
\end{proposition}

\begin{proof}
We begin with a few elementary observations, relating information about $L^\perp$ to information about the symplectic complement $L^\pperp$. Using the explicit forms of $\theta$ and $Q$ as in \eqref{eq:forms} one can check directly that, apriori only as non-oriented planes, we have  
\begin{equation}\label{eq:JisOrtho}
        L^\pperp=\pa{L\Notj}^\perp
        =L^\perp \overline{\Notj}=L^\perp (-\Notj)=L^\perp \Notj
    \end{equation}  
    where we view $L,L^\perp,L^\pperp$ as subsets of~$\quat(\Z)$.
    We further have for $v,w\in\quat(\Z)$
    \begin{equation}\label{eq:PPvsPSympForm}
    \theta(v\Notj,w\Notj)=\trace\pa{v\Notj\Notj\overline{\Notj}\overline{w}}=-\trace\pa{v\Notj\overline{w}}=-\theta(v,w).
    \end{equation}

Recall that $L^\perp = \Psi(-a_1(L),a_2(L))$ and so $L^\perp$ is symplectic by Lemma~\ref{lem:SympIf1}. 
Choose then a symplectic basis $(w_1,w_2)$ for $L^\perp$ and notice that $(w_2\Notj,w_1\Notj)$ is a (a priori non-oriented) basis of $L^\pperp$ by \eqref{eq:JisOrtho}. 
By \eqref{eq:PPvsPSympForm}, $\theta(w_2\Notj,w_1\Notj)=-\theta(w_2,w_1)=1$ so it is a symplectic basis. We can therefore use it to calculate the quadratic form on $L^\pperp$ and obtain
\begin{align*}
[q_{L^\pperp}] = [\det(x w_2\Notj + y w_1 \Notj)] = [-\det(xw_2+yw_1)] 
= [-q_{(L^\perp)_{\mathrm{op}}}]
= [-\overline{q_{L^\perp}}].
\end{align*}
In particular, $L^\pperp$ has (non-zero) discriminant equal to $D$.

It remains to compute the Klein vectors. We use the basis $(w_2\Notj,w_1\Notj)$ again and obtain from the definition of Klein vectors in \eqref{eq:DefOfKleinMap} and the orientation on the orthogonal complement in \eqref{eq:orthcompl-oriented} that
\begin{align*}
    a_1(L^\pperp)&=2w_2\Notj\overline{w_1\Notj}-\trace\pa{w_2\Notj\overline{w_1\Notj}}I_2
    =-2w_2\overline{w_1}+\trace\pa{w_2\overline{w_1}}I_2\\
    &=-\overline{a_1(L^\perp)}=a_1(L^\perp)=-a_1(L).
\end{align*}
Similarly,
\begin{align*}
    a_2(L^\pperp)&=2\overline{w_1\Notj}w_2\Notj-\trace\pa{\overline{w_1\Notj}w_2\Notj}I_2
    =-2\Notj\overline{w_1}w_2\Notj+\trace\pa{\Notj\overline{w_1}w_2\Notj}I_2\\ 
    &=-\Notj\overline{a_2(L^\perp)}\Notj=\Notj{a_2(L^\perp)}\Notj=\SmMat{1}{-\alpha}{-\gamma}{-1},
\end{align*}
proving the proposition.
\end{proof}

\section{Gauss composition via Klein vectors}\label{sec:Kleinrelations}

In this section, we establish the expression for Gauss composition in terms of Klein vectors as announced in the introduction in Theorem~\ref{thm:gausscompositionviaKleinvectors}.

To understand how the Klein vectors $a_1(L),a_2(L)$ determine the quadratic form $[q_L]$, we identify them with IBQFs via the map $A \in \Lambda \mapsto q_A$ from \S\ref{sec:qf vs traceless}.
Two vectors in the Gross lattice $\Lambda$ are pair-primitive (cf.~Definition~\ref{def:pairprimitive}) if and only if the associated IBQFs have coprime content.
Recall also that $\Lcal_{\neq 0}$ denotes the set of planes $L$ with $\disc(q_L)$ non-zero.

The following theorem is a slightly more precise version of Theorem~\ref{thm:gausscompositionviaKleinvectors} from the introduction.

\begin{theorem}\label{thm:nonprimthm}
For any $L=L_{a_1,a_2} \in \mathcal{L}_{\neq 0}$ with $a_1, a_2\in \Lambda$ pair-primitive we have
\begin{align*}
      [q_L] = \overline{[q_{a_1}]}*[q_{a_2}].
\end{align*}
Moreover, the content of $[q_L]$ is the product of the contents of $q_{a_1},q_{a_2}$.
In particular, if $a_1,a_2$ are primitive, then $[q_L] = [q_{a_1}]^{-1}*[q_{a_2}]$ is primitive.
\end{theorem}

\begin{proof}
Using equivariance of the Klein map $\Phi$ (see Proposition~\ref{prop:KleinBasicProp}\ref{item:Phiequiv}) as well as the map $A \mapsto q_A$ (see \S\ref{sec:qf vs traceless}) we may bring $\overline{q_{a_1}},q_{a_2}$ into concordant form (cf.~\S\ref{sec:Dirichlet composition}) without changing $[q_L]$.
That is, we suppose that $\overline{q_{a_1}} = \alpha_1 x^2 + \beta xy + \gamma_1 y^2$ and $q_{a_2} = \alpha_2 x^2 + \beta xy + \gamma_2 y^2$, where $\alpha_1\neq 0\neq \alpha_2$ satisfy $\alpha_1 \mid \gamma_2$ and $\alpha_2 \mid \gamma_1$.
Using the explicit inverse \eqref{eq:nielK} of the Klein map in Proposition~\ref{prop:KleinBasicProp} and disregarding orientation at first, we have
\begin{align*}
L=\set{x\in \quat(\Z): \SmMat{-\beta}{-2\alpha_1}{2\gamma_1}{\beta}x=x\SmMat{\beta}{-2\alpha_2}{2\gamma_2}{-\beta}}.
\end{align*}
It is readily checked that
\begin{align*}
v_1 = \SmMat{-\alpha_1}{0}{\beta}{-\alpha_2},\
v_2 = \SmMat{0}{-1}{\frac{\gamma_2}{\alpha_1}}{0}
\end{align*}
is an oriented basis of $L$.
Since $\det(xv_1+yv_2) = \alpha_1\alpha_2 x^2 + \beta xy + \star y^2$, the definition of the Gauss composition via concordant forms implies that $[q_L]$ is indeed the composition of $\overline{[q_{a_1}]}$ and $[q_{a_2}]$.
The statement regarding the content follows from Proposition~\ref{prop:dirichlet law}.
\end{proof}

\subsection{Consequences for orthogonal and symplectic complements}\label{sec:qf on complements}

We now deduce expressions for the class of the IBQFs on the orthogonal and symplectic complements.

\begin{cor}\label{cor:nonprimLperpLpperp}
 For any $L=L_{a_1,a_2} \in \mathcal{L}_{\neq 0}$ we have $[q_{L^\perp}] = \overline{[-q_{a_1}]}*[q_{a_2}]$. Furthermore, if $L$ is symplectic,
    \begin{align*}
        [q_{L^\pperp}] &= [-\overline{q_{L^\perp}}] =  [q_{a_1}]*[q_{a_2}].
\end{align*}
Moreover, for primitive $a\in \Lambda$, we have
$[q_{\langle \id, a\rangle^\perp}] = -[q_{a}]^2$.
\end{cor}

In Corollary~\ref{cor:nonprimLperpLpperp}, $\langle \id, a\rangle^\perp$ is taken to be (the integer points of) the orthogonal complement of the plane $\langle \id, a\rangle \cap \quat(\Z)$ with orientation given by $\id,a$.

\begin{proof}
The formula for $[q_{L^\perp}]$ 
follows from applying Theorem~\ref{thm:nonprimthm} to $L^\perp
{=}L_{-a_1, a_2}$ (see \eqref{eq:orthcompl-oriented}) and using $q_{-a_1} = -q_{a_1}$.
The formulas for $[q_{L^\pperp}]$ are obtained from Proposition~\ref{prop:KleinForSymp} and Lemma~\ref{lem:flipthingsaround}.
Alternatively, one may apply the formulas in Proposition~\ref{prop:KleinForSymp} for $a_1(L^\pperp)$ and $a_{2}(L^\pperp)$ in conjunction with Theorem~\ref{thm:nonprimthm}.
The last equality follows from $\langle \id,a\rangle^\perp = (L_{-a,-a})^\perp = L_{a,-a}$ (cf.~\eqref{eq:orthcompl-oriented}) and Lemma~\ref{lem:flipthingsaround}b).
\end{proof}


For a symplectic plane $L=L_{a_1,a_2}$ we have by Lemma \ref{lem:SympIf1}  that $a_2$ is  automatically primitive and $[q_{a_2}]\in \mathcal{G}_D^+$. We therefore have the following.

\begin{cor}\label{cor:stabilizer}
    For a symplectic plane $L=L_{a_1,a_2}$ we have 
    \[
[q_L]=[q_{L^\pperp}]\iff [q_{a_1}]=\overline{[q_{a_1}]},
    \]
    and
    \begin{align}\label{eq:stabcor-2}
    [q_L]=\overline{[q_{L^\pperp}]}\iff [q_{a_2}]^2\in \Stab_{\mathcal{G}_D^+}([q_L])\iff 
-[q_{\langle \id,a_2\rangle^\perp}]\in \Stab_{\mathcal{G}_D^+}([q_L]),
    \end{align}
where $\Stab_{\mathcal{G}_D^+}([q_L])$ is the set of elements $g$ in the group ${\mathcal{G}_D^+}$ such that $g[q_L]=[q_L]$.
In particular, if $q_L$ is primitive then
\begin{align*}
[q_L] \neq [q_{L^\pperp}] \iff [q_{a_1}]^2\neq 1,\quad
[q_L]\neq \overline{[q_{L^\pperp}]} \iff [q_{a_2}]^2\neq 1.
\end{align*}
\end{cor}

For completeness, we use the above corollary to prove Corollary~\ref{cor:symplecticplanes} from the introduction. 
While we do not directly apply Corollary~\ref{cor:symplecticplanes} in our application to Seifert surfaces, the proof of~Theorem~\ref{thm:main3d} uses similar ideas and is also based on Corollaries~\ref{cor:nonprimLperpLpperp}, \ref{cor:stabilizer}.

\begin{proof}[Proof of Corollary~\ref{cor:symplecticplanes}]
Let $D$ be an integer with $D \equiv 1 \mod 4$.
Suppose first that there exists a symplectic plane $L \subset \Z^4$ such that $q_L$ has discriminant $D$ and $[q_L] \ne [q_{L^\pperp}], \overline{[q_{L^\pperp}]}$.
By Lemma~\ref{lem:SympIf1}, we have $q_{a_2(L)}=ax^2+xy+cy^2$ for some
$a,c\in\Z$ with $1-4ac=D$.
If $[q_{a_2(L)}]^2=1$, Corollary~\ref{cor:stabilizer} (specifically \eqref{eq:stabcor-2}) would imply
$[q_L]=\overline{[q_{L^{\pperp}}]}$, a contradiction.

Conversely, suppose that $q=ax^2+xy+cy^2$ has discriminant $D$
and satisfies $[q]^2\ne 1$.
Since $q$ is primitive, Proposition~\ref{prop:KleinBasicProp} and Lemma~\ref{lem:SympIf1} provide
a symplectic plane $L\in\mathcal L_D$ with
$q_{a_1(L)}=q_{a_2(L)}=q$.
The form $q_L$ is primitive by Theorem~\ref{thm:nonprimthm}, so Corollary~\ref{cor:stabilizer}
gives $[q_L]\ne [q_{L^{\pperp}}],
\overline{[q_{L^{\pperp}}]}$, as desired.
\end{proof}

\begin{remark}[Relations of IBQFs]\label{rem:relations}
In the above (Theorem~\ref{thm:nonprimthm} and Corollary~\ref{cor:nonprimLperpLpperp}), we have established exact relations in the six tuple
\begin{align}
\big([q_{a_1}],[q_{a_2}],&[q_{L}],[q_{L^\perp}],[q_{\langle \id,a_1\rangle^\perp}],[q_{\langle \id,a_2\rangle^\perp}]\big)\nonumber\\
&= \big([q_{a_1}],[q_{a_2}],\overline{[q_{a_1}]}*[q_{a_2}],\overline{[-q_{a_1}]}*[q_{a_2}],-[q_{a_1}]^2,-[q_{a_2}]^2\big)\label{eq:fullrelations}
\end{align}
for any $L = L_{a_1,a_2}\in \Lcal_D$ with $a_1,a_2$ primitive.
These are the tuples considered in the equidistribution results of \cite{AEW-2in4}.
\end{remark}

\subsection{On the relation to the Bhargava cube}\label{sec:Bhargava}

In the centuries since Gauss's original publication of his composition law in the {\em Disquisitiones}, many other mathematicians have revisited it with the aim of elucidating the fundamental underlying structure and obtaining more powerful generalizations.  In recent years, the most influential reformulation of Gauss composition has been the ``cube law'' due to Manjul Bhargava in his paper \cite{bhargava_composition}, where he used it as the launching point to define a range of other composition laws.

In this section, we state Bhargava's cube law and compare it to our formulation of Gauss composition above.
We note that these discussions will not be used in the remainder of this paper.

Given a cube $\mathsf{A}$ with integers at the vertices
\[
\begin{tikzcd}
  & e \arrow[dash, dl] \arrow[dash, dd] \arrow[dash, rr] &   & f \arrow[dash, dl] \arrow[dash, dd]\\
a \arrow[dash, dd] \arrow[dash, rr, crossing over] &   & b & \\
& g \arrow[dash, dl] \arrow[dash, rr] &    & h \arrow[dash, dl] \\
c \arrow[dash, rr] &  &  d \arrow[dash, uu, crossing over] &
\end{tikzcd}
\]
one can slice it in three ways to make pairs of matrices
\begin{equation}\label{eq:MiNi}
\begin{split}
M_1 = & \begin{pmatrix} a & b \\ c & d \end{pmatrix}, \quad N_1 = \begin{pmatrix} e & f \\ g & h \end{pmatrix}, \\ 
M_2 = & \begin{pmatrix} a & c \\ e & g \end{pmatrix}, \quad N_2 = \begin{pmatrix} b & d \\ f & h \end{pmatrix},\\
M_3 = & \begin{pmatrix} a & e \\ b & f  \end{pmatrix}, \quad N_3 = \begin{pmatrix} c & g \\ d & h \end{pmatrix}.
\end{split}
\end{equation}
For each pair of matrices $M_i$, $N_i$, one defines a binary quadratic form by
\begin{equation}\label{eq:BhargavaQuadForm}
q_i^{\mathsf{A}}(x, y) = -\det(M_i x - N_i y).
\end{equation}
This construction intertwines the actions of $\SL_2(\Z) \times \SL_2(\Z) \times \SL_2(\Z)$ on cubes of integers (viewed as elements of the external tensor product representation $\Z^2 \otimes \Z^2 \otimes \Z^2$) and on triples of binary quadratic forms (where the first copy of $\SL_2(\Z)$ acts on the first quadratic form and similarly for the second and third copy).

Comparing with \eqref{eq:QuadFormCoeff}, we see that 
\[
q_i^{\mathsf{A}} = -\overline{q_{L_i}}
\]
if $(M_i, N_i)$ spans a plane $L_i$ over $\Z$.

 Bhargava's version of the composition law is typically stated for classes of primitive forms only, where it is most symmetric.

\begin{proposition}[{Bhargava's formulation of Gauss composition (\cite[Thm.~1]{bhargava_composition})}]\label{prop:bhargava} 
The three IBQFs $q_1^{\mathsf{A}}, q_2^{\mathsf{A}}, q_3^{\mathsf{A}}$ defined in \eqref{eq:BhargavaQuadForm} all have the same discriminant $D$.  
If $q_1^{\mathsf{A}}, q_2^{\mathsf{A}}, q_3^{\mathsf{A}}$ are all primitive, their classes in $\mathcal{G}_D^+$ are related by
\begin{equation}\label{eq:bhargavacomposition}
[q_1^{\mathsf{A}}] * [q_2^{\mathsf{A}}] *[q_3^{\mathsf{A}}] = 1 \in \mathcal{G}_D^+.
\end{equation}
In particular, any two of the classes $[q_1^{\mathsf{A}}]$, $[q_2^{\mathsf{A}}]$, $[q_3^{\mathsf{A}}]$ determine the third.
Conversely, any triple of IBQFs $(Q_1, Q_2, Q_3)$ satisfying \eqref{eq:bhargavacomposition} can be realized as $(q_1^{\mathsf{A}}, q_2^{\mathsf{A}}, q_3^{\mathsf{A}})$ for some cube ${\mathsf{A}}$ of integers.
\end{proposition}

As we have seen in \S\ref{sec:Dirichlet composition}, Dirichlet composition is also well defined for pairs of non-primitive quadratic forms, as long as the two quadratic forms have equal discriminant and coprime contents.  
Bhargava's viewpoint on Gauss composition can be phrased in this context as well, we omit it here for simplicity.
With this generalization at hand, one may give an alternative proof of Theorem~\ref{thm:nonprimthm}.

Conversely, Theorem~\ref{thm:nonprimthm} provides a different perspective on Bhargava's cube law.
To illustrate this, we shall prove Proposition~\ref{prop:bhargava} from Theorem~\ref{thm:nonprimthm}.

\begin{proof}[Proof of Proposition~\ref{prop:bhargava}]
Given a cube $\mathsf{A}$ as in the statement, let $M_3,N_3$ as in \eqref{eq:MiNi} and let $L$ be the plane containing $M_3,N_3$ and with orientation matching $(M_3,N_3)$.
By definition, we have
\begin{align*}
q_3^{\mathsf{A}} = -\det(M_3 x -N_3 y) = -\overline{q_L}(x,y)
\end{align*}
if $M_3,N_3$ is a basis of $L$.
We let $\lambda$ be the index of $\Z M_3+ \Z N_3$ in $L$.
By a direct calculation in coordinates for $M_3,N_3$ (i.e.~in the entries of the cube) we have
\begin{align*}
q^{\mathsf{A}}_1=\SmMat{0}{1}{1}{0}.q_{a_2'},\qquad 
q^{\mathsf{A}}_2=-\overline{q_{a_1'}}.
\end{align*}
where $a_2' = 2 M_3\overline{N_3}- \trace(M_3\overline{N_3})I_2$ and $a_2' = 2\overline{N_3} M_3- \trace(\overline{N_3}M_3)I_2$.
But $a_i' = \lambda a_i(L)$ for $i=1,2$ and, since $q_1^{\mathsf{A}}, q_2^{\mathsf{A}}$ are primitive, we have $\lambda = 1$ and $a_i' = a_i(L)$ for $i=1,2$.
In particular, $M_3,N_3$ is a basis of $L$ and 
\begin{align*}
q^{\mathsf{A}}_1=\SmMat{0}{1}{1}{0}.q_{a_2(L)},\qquad 
q^{\mathsf{A}}_2=-\overline{q_{a_1(L)}},\qquad
q_3^{\mathsf{A}} =  -\overline{q_L}.
\end{align*}

By Theorem~\ref{thm:nonprimthm}, we have $[q_L]= \overline{[q_{a_1(L)}]} \ast [q_{a_2(L)}] = [-q_2^{\mathsf{A}}]\ast \overline{[q_1^{\mathsf{A}}]}$.
Using Lemma~\ref{lem:flipthingsaround} in addition, we obtain $[q_3^{\mathsf{A}}] = [-\overline{q_L}] = \overline{[q_1^{\mathsf{A}}]}\ast \overline{[q_2^{\mathsf{A}}]}$ as claimed.

For the converse, let $(Q_1, Q_2, Q_3)$ be any triple of IBQFs satisfying \eqref{eq:bhargavacomposition}.
Let $a_1 = -A(\overline{Q_2})$, $a_2 = A(\SmMat{0}{1}{1}{0}.Q_1)$ and consider the plane $L = L_{a_1,a_2}$.
Let $v_1,v_2$ be a basis of $L$ and consider the cube ${\mathsf{A}}$ with $M_3=v_1,\,N_3=v_2$ in \eqref{eq:MiNi}.
By the same computations as above, we have
\begin{align*}
q^{\mathsf{A}}_1=\SmMat{0}{1}{1}{0}.q_{a_2} = Q_1,\qquad 
q^{\mathsf{A}}_2=-\overline{q_{a_1}} = Q_2,\qquad q_3^{\mathsf{A}} = -\overline{q_L}.
\end{align*}
Theorem~\ref{thm:nonprimthm} and Lemma~\ref{lem:flipthingsaround} show that $[q_3^{\mathsf{A}}] = [-\overline{q_L}] = \overline{[Q_1]} \ast \overline{[Q_2]} = [Q_3]$.
After an $\SL_2(\Z)$-change of basis in $L$, we may thus assume that $q_3^{\mathsf{A}} = Q_3$. As $a_1(L),a_2(L)$ are independent of the choice of basis, so are $q_1^{\mathsf{A}},q_2^{\mathsf{A}}$ and we conclude that $q_i^{\mathsf{A}} = Q_i$ for $i=1,2,3$ as desired.
\end{proof}

\section{Knot theory: realization of disjoint Seifert surfaces with prescribed Seifert forms}\label{sec:knottheorysetup}
In this section we set up the relevant notions from knot theory and prove Theorems~\ref{thmintro:main3d} and~\ref{thmintro:main4d}. In \S\ref{ss:setupSF} we recall the notion of Seifert forms and check that for closed surfaces $\Sigma\subset S^3$ of genus $2$ the Seifert form on $H_1(\Sigma;\Z)\cong\Z^4$ is, up to a choice of coordinates, the bilinear form $(\Z^4, U)$ studied in \S\ref{sec:2in4Matrices}. This is the reason that the form $U$ has relevance to the Seifert surface questions outlined in the beginning of the introduction.
In \S\ref{ss:realizationofSFs} we show that $2$-dimensional symplectic planes $L$ of $H_1(\Sigma;\Z)$ (with respect to the intersection form) correspond to the first homology of a genus $1$ surface $F \subset \Sigma$ whose boundary is a knot.
Moreover, the  symplectic complement $L^\pperp$ of $L$ is equal to the first homology of the complement $\Sigma \setminus F$.
We then show how the results about symplectic planes in $(\Z^4, U)$ and their complements from \S\ref{sec:2in4Matrices} imply the Seifert form realization result as stated in Theorem~\ref{thmintro:main3d}.
Finally, in \S\ref{ss:obstructionstoisotopyinB4}, we describe the intersection form on second homology of the double branched cover of Seifert surfaces that are pushed into $\B^4$. We then use the $3$-dimensional Seifert form realization results to establish the existence of many pairs of genus 1 Seifert surfaces that are not ambiently isotopic after being pushed into $\B^4$. 
In particular, we prove Theorem~\ref{thmintro:main4d}.
\subsection{Setup: Seifert form and Seifert surfaces}\label{ss:setupSF}
 We point to~\cite[Chap.~6]{Lickorish_97} for details on the Seifert form.
For an oriented compact smooth surface with (possibly empty) boundary $S\subset S^3$, we denote by \[U_S\colon H_1(S;\Z)\times H_1(S;\Z)\to \Z\] the Seifert form. Representing elements in $H_1(S;\Z)$ as curves, the Seifert form may be defined by $U_S([\alpha],[\beta])=\mathrm{lk}(\alpha,\beta^+)$. Here, $\mathrm{lk}$ denotes Gauss's linking number of disjoint curves in $S^3$, $\alpha$ and $\beta$ are curves in $S$ representing a given homology class, and $\beta^+$ is the result of pushing $\beta$ a small amount along a positive normal into $S^3\setminus S$; see the right-hand side of Figure~\ref{fig:Sigmaastd}. 

\begin{figure}[ht]
  \centering

  \def\svgscale{0.9}
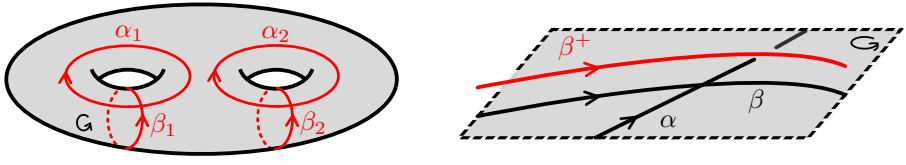
  \caption{Left: the surface $\Sigma_{\rm std}$ (gray) with simple closed curves $\alpha_1,\beta_1,\alpha_2,\beta_2$ (red) such that the corresponding homology classes form a basis. Right: Illustration of the positive push-off $\beta^+$ (red) of $\beta$ in $\Sigma$ (gray) needed for the definition of the Seifert form.}\label{fig:Sigmaastd}
\end{figure}

Writing $\overline{S}$ for the surface $S$ with reversed orientation, we observe that $U_{\overline{S}}(a,b) = U_S(b, a)$ for all $a, b \in H_1(S; \Z)$. Let $\theta_S$ denote the intersection pairing on $H_1(S; \Z)$; that is, if $\alpha$ and $\beta$ are transverse curves in $S$, $\theta_S([\alpha], [\beta])$ is the total number of signed intersections of $\alpha$ and $\beta$. By Poincaré duality, $\theta_S$ is a perfect pairing when $S$ has at most $1$ boundary component. 
It is well-known (see, e.g.~the proof of  \cite[Theorem~6.10]{Lickorish_97}) that $\theta_S$ is the antisymmetrization of $U_S$, that is,
$
\theta_S(a, b) =U_S(a,b)-U_S(b,a) = U_S(a, b) - U_{\overline{S}}(a, b).
$
We denote by $Q_S$ the symmetrization of the Seifert form,
i.e.~$Q_S\colon H_1(S;\Z)\times H_1(S;\Z)\to \Z, (a,b)\mapsto U_S(a,b)+U_S(b,a)$.
We choose the notation $U_S$ for the Seifert form due to the following.
\begin{example}\label{ex:stdsurface}
Let $\Sigma_{\rm std}$ be the unknotted, oriented genus two closed subsurface of $S^3$ as depicted in Figure~\ref{fig:Sigmaastd}.
Then there exists a basis $\mathcal{B}=\left([\alpha_1],[\beta_1],[\alpha_2],[\beta_2]\right)$ of $H_1(\Sigma_{\rm std};\Z)$, where $\alpha_i$ and $\beta_i$ are as indicated in Figure~\ref{fig:Sigmaastd}, such that with respect to this basis the Seifert form is represented by the matrix
$U=\big(\begin{smallmatrix}
  0 & 1\\
  0 & 0
\end{smallmatrix}\big)\oplus\big(\begin{smallmatrix}
  0 & 1\\
  0 & 0
\end{smallmatrix}\big)$ from~\eqref{eq:BilFormU}.
\end{example}
In fact, every closed oriented surface of genus $g$ has as its Seifert form the $g$-fold direct sum $\big(\begin{smallmatrix}
  0 & 1\\
  0 & 0
\end{smallmatrix}\big)\oplus \cdots \oplus \big(\begin{smallmatrix}
  0 & 1\\
  0 & 0
\end{smallmatrix}\big)$.

\begin{lem}\label{lem:SFforclosedsurfaces}
Let $\Sigma$ be an oriented closed genus $g$ surface in $S^3$. 
Then there is a basis of $H_1(\Sigma;\Z)$ in which $U_\Sigma$ is represented by  $$\underbrace{\big(\begin{smallmatrix}
  0 & 1\\
  0 & 0
\end{smallmatrix}\big)\oplus \cdots \oplus \big(\begin{smallmatrix}
  0 & 1\\
  0 & 0
\end{smallmatrix}\big)}_{g}.$$
\end{lem}
\begin{proof} 
Write $S^3=X\cup Y$, where $X$ and $Y$ are 3-manifolds such that $X\cap Y=\Sigma$ (as sets) and such that $X$ and $Y$ have oriented boundary $\Sigma$ and $ \overline{\Sigma}$, respectively. 
In particular, if $\gamma$ is a curve in $\Sigma$, then $\gamma^+$ is in the interior of $Y$.  Likewise, if we analogously define $\gamma^-$ as the curve given by pushing $\gamma$  along a negative normal, $\gamma^-$ is contained in the interior of $X$.
\par
Applying Mayer-Vietoris to this decomposition of $S^3$ gives
\[
H_2(S^3; \Z)  \to H_1(\Sigma; \Z)\xrightarrow{\cong} H_1(X;\Z)\oplus H_1(Y;\Z) \to H_1(S^3; \Z),
\]
where the middle map is induced by 
inclusions, and is an isomorphism since the outer terms are both 0.  

Write 
\begin{align*}
A & 
= \ker(H_1(\Sigma; \Z) \to H_1(Y; \Z)), \\
B & 
=\ker(H_1(\Sigma; \Z) \to H_1(X; \Z)).
\end{align*}
Via the Mayer-Vietoris isomorphism above, $A$ and $B$ are the preimages of $H_1(X;\Z)\oplus 0$ and $0\oplus H_1(Y;\Z)$, respectively. Therefore, $A \oplus B$ equals $H_1(\Sigma; \Z)$ and, in particular, $A$ and $B$ are free and their ranks sum to $2g$. (Below we will furthermore observe that $\rank A = \rank B = g$, an instance of the so-called ``half lives, half dies'' principle; compare e.g.~\cite[Lemma~3.5]{hatcher2023notes}.)

Next, we will show geometrically that, for all $a_1, a_2 \in A$ 
\begin{equation}\label{eq: (a_1, a_2) = 0}
U_{\Sigma}(a_1, a_2) = 0,
\end{equation}
while for all $b \in B$ and all $c \in H_1(\Sigma; \Z)$,
\begin{equation}\label{eq: (b, c) = 0}
U_{\Sigma}(b, c) = 0.
\end{equation}
Since it suffices to consider primitive elements and since primitive elements in $H_1(\Sigma;\Z)$ are representable by  simple closed curves, we may and do assume that $a_1 = [\alpha_1]$, $a_2 = [\alpha_2]$, $b = [\beta]$ and $c = [\gamma]$ are classes of simple closed curves in $\Sigma$, where $\alpha_1, \alpha_2$ are null-homologous in $Y$ and $\beta$ is null-homologous in $X$.

The vanishing properties \eqref{eq: (a_1, a_2) = 0} and \eqref{eq: (b, c) = 0} follow from the fact that if $\gamma_1, \gamma_2$ are disjoint curves in $S^3$ such that $\gamma_1$ bounds a surface $T \subset S^3$ that is disjoint from $\gamma_2$, then $\mathrm{lk}(\gamma_1, \gamma_2) = \#(T \cap \gamma_2) = 0$.  Indeed, to get \eqref{eq: (a_1, a_2) = 0}, we observe that  $U_{\Sigma}(a_1, a_2) = \mathrm{lk}(\alpha_1, \alpha_2^+) = \mathrm{lk}(\alpha_2, \alpha_1^-)$, and $\alpha_2$ bounds a surface contained in $Y$ while $\alpha_1^{-}$ is in the interior of $X$.  Likewise for \eqref{eq: (b, c) = 0}, $U_{\Sigma}(b, c) = \mathrm{lk}(\beta, \gamma^+)$, and $\beta$ bounds a surface contained in $X$ while $\gamma^+$ is contained in the interior of $Y$. 

We are now reduced to algebra.  By the above, $A$ and $B$ are complementary subgroups that are isotropic for the intersection pairing $\theta_\Sigma(x, y) = U_\Sigma(x, y) - U_\Sigma(y, x)$, which is unimodular by Poincar\'e duality. Noting that an isotropic subgroup of a unimodular symplectic lattice has rank at most half that of the lattice, we improve $\rank A+\rank B=2g$ to $\rank A=\rank B=g$.  Choose a basis $(e_1, \dotsc, e_g)$ of $A$. Using that $\theta_\Sigma$ induces an isomorphism $B \to \mathrm{Hom}(A;\Z)$, we have a basis $(f_1, \dotsc, f_g)$ of $B$ such that $\theta_\Sigma(e_i, f_j) = \delta_{i, j}$.

Applying \eqref{eq: (a_1, a_2) = 0} and \eqref{eq: (b, c) = 0}, we conclude that the basis $(e_1, f_1, \dotsc, e_g, f_g)$ has the desired properties that, for all $i, j$
\[
U_\Sigma(e_i, e_j) = U_\Sigma(f_i, e_j) = U_\Sigma(f_i, f_j) = 0
\]
while
\[
U_\Sigma(e_i, f_j)=\theta_\Sigma(e_i, f_j) + U_\Sigma(f_j, e_i)=\delta_{i, j}.\qedhere
\]
\end{proof}

Given Lemma~\ref{lem:SFforclosedsurfaces}, it is no surprise that most commonly the Seifert form is considered for surfaces with non-empty boundary. A \emph{Seifert surface} $F$ (for a knot or link $K$) is an oriented compact non-empty smooth surface in $S^3$ without closed components (with oriented boundary $K$). If $F$ is a genus $1$ Seifert surface for a knot $K$, we write
\begin{align}\label{eq:s_F}
s_F = [U_F(a,a)x^2+(U_F(a,b)+U_F(b,a))xy +U_F(b,b)y^2]\in \mathcal{Q}^+_D
\end{align}
where $(a,b)$ is a symplectic basis for $H_1(F;\Z)$ (i.e.~a basis with respect to which the intersection form is given by $\big(\begin{smallmatrix}
  0 & 1\\
  -1 & 0
\end{smallmatrix}\big)$) and $D$ is the discriminant $(U_F(a,b)+U_F(b,a))^2-4U_F(a,a)U_F(b,b)$.
Note that $s_{\overline{F}} = \overline{s_F}$ since $(a,b)$ is a symplectic basis for  $H_1(F;\Z)$ if and only if  $(a,-b)$ is a symplectic basis for  $H_1(\overline{F};\Z)$.

\subsection{Realization of Seifert forms and proof of Theorem~\ref{thmintro:main3d}}\label{ss:realizationofSFs}
For an oriented closed smooth genus 2 surface $\Sigma \subset S^3$, we define a \emph{separating knot $K \subset \Sigma$} to be an oriented simple closed curve $K \subset \Sigma$ that is separating and non-contractible in $\Sigma$.  We can also think of such a $K$ as a knot inside $S^3$ equipped with two Seifert surfaces that are (the closures of) the two connected components of $\Sigma \setminus K$.  Since $K$ is non-contractible, these two components must both have genus $1$.  To distinguish them,  let  $ F_K$ be the closure of the component of $\Sigma\setminus K$ that induces the correct orientation on $K$ and $F_K^\pperp$ be the closure of the other component of $\Sigma\setminus K$.

We now give some lemmas relating this construction to the algebra of planes in $H_1(\Sigma; \Z) \cong \Z^4$.

\begin{lem}\label{lem:realizationofL} 
Let $\Sigma$ be an oriented closed smooth genus 2 surface in $S^3$. For every symplectic plane $L$ of $H_1(\Sigma;\Z)$, there exists a separating knot $K\subset \Sigma$ such that $L=H_1(F_K;\Z)\subset H_1(\Sigma;\Z)$.
\end{lem}

\begin{proof}
Let $(a,b)$ be a symplectic basis of $L$; that is, $a$ and $b$ are primitive and satisfy $\theta(a,b)=1$. There exist simple closed curves $\alpha$ and $\beta$ that intersect once transversely such that $a=[\alpha]$ and $b=[\beta]$ (e.g.~by \cite[Theorem~6.4]{FarbMargalit_12_APrimerOnMCG}). Consider a closed deformation retraction neighborhood $N$ of  $\alpha\cup\beta$ that is a smooth surface with boundary. $N$ is diffeomorphic to a genus 1 surface with one boundary component. Set $K\coloneqq\partial N$ and $F_K\coloneqq N$. By construction $([\alpha],[\beta])$ forms a basis of $H_1(F_K;\Z)$, hence $L=H_1(F_K;\Z)$.
\end{proof}

\begin{lem}\label{lem:Lpperp=H1(FKppepr)}
Let $K$ be a separating knot in a closed oriented smooth genus 2 surface $\Sigma$. Then $H_1(F_K;\Z)\subset H_1(\Sigma;\Z)$ is a symplectic plane and $H_1(F_K;\Z)^\pperp=H_1(F_K^\pperp;\Z)\subset H_1(\Sigma;\Z)$.
\end{lem}

\begin{proof}
We write $F$ for $F_K$ and $F^\pperp$ for $F_K^\pperp$.  By the Mayer-Vietoris exact sequence, the natural map $H_1(F; \Z) \oplus H_1(F^{\pperp}; \Z) \to H_1(\Sigma; \Z)$ has trivial cokernel (as $F \cap F^{\pperp}$ is connected and hence the connecting homomorphism $H_1(\Sigma;\Z) \to H_0(F\cap F^{\pperp};\Z)$ is trivial) and trivial kernel (as the map $H_1(F \cap F^{\pperp}; \Z) \to H_1(F; \Z) \oplus H_1(F^{\pperp}; \Z)$ has trivial image).  Hence we have the direct sum decomposition 
\begin{equation}\label{eq:topological direct sum}
 H_1(\Sigma; \Z) = H_1(F; \Z) \oplus H_1(F^{\pperp}; \Z).
 \end{equation}
 
We observe that the restriction of $\theta_\Sigma$ to $H_1(F; \Z)$ is the intersection form on $H_1(F; \Z)$, which we denote by $\theta_F$.  Since $F$ is a genus $1$ surface with $1$ boundary component, this form has matrix $\left( \begin{smallmatrix} 0 & 1 \\ -1 & 0 \end{smallmatrix} \right)$ with respect to an appropriate basis, and so $H_1(F;\Z)$ is symplectic.  

It remains to check that $H_1(F^{\pperp}; \Z) = (H_1(F; \Z))^{\pperp}$.  The inclusion  $H_1(F^{\pperp};\Z)  \subset (H_1(F; \Z))^{\pperp}$ holds because any element of $H_1(F^{\pperp};\Z)$ can be represented by a cycle in the interior of $F^{\pperp}$, which will have zero intersection with any cycle in $F$. By the direct sum decomposition \eqref{eq:topological direct sum} and since $H_1(F;\Z)\cap H_1(F;\Z)^{\pperp}=0$, we must in fact have equality $H_1(F^{\pperp};\Z) = H_1(F;\Z)^{\pperp}$.
\end{proof}

  Note that $\partial F_K = K$ while $\partial F_K^{\pperp} = -K$ is the knot given by reversing the orientation of $K$.  Since we are working with oriented knots, $F_K$ and $F_K^{\pperp}$ are not Seifert surfaces for the same knot.  However, we can fix this with a change of orientation, as below.

  For a separating knot $K$ inside a genus $2$ surface $\Sigma \subseteq S^3$, we define $F_1= F_1(K) = F_K$ and $F_2 = F_2(K) = \overline{\Sigma\setminus (F_K\setminus K)}$.  These are now both Seifert surfaces for $K$.  Comparing with our old definition, we observe that $F_2 = \overline{F_K^{\pperp}}$ is the surface given by reversing the orientation on $F_K^{\pperp}$.

The following theorem is a slightly more precise version of Theorem~\ref{thmintro:main3d}.

\begin{theorem}\label{thm:main3d}
Fix a closed oriented smooth genus $2$ surface $\Sigma\subseteq S^3$ and fix $D \equiv 1 \pmod{4}$. Let $s_1$ and $s_2$ be in $\mathcal{Q}^+_D$. The following are equivalent:
\begin{enumerate}[(a)]
    \item \label{item:KinSigma}
There exists a separating knot $K\subseteq \Sigma$ with $D_K = D$ such that the Seifert surfaces $F_1\coloneqq F_K$ and $F_2\coloneqq \overline{\Sigma\setminus (F_K\setminus K)}$ satisfy $s_1=s_{F_1}$ and $s_2=s_{F_2}$.
\item \label{item:t^2s1=s2}
There exist $a,c\in \Z$ with $1-4ac=D$ such that $[ax^2+xy+cy^2]^2*s_1=s_2$.
\end{enumerate}
\end{theorem}


\begin{proof}
In this proof we identify $H_1(\Sigma;\Z)$ with $\Z^4$ using the basis described in Lemma~\ref{lem:SFforclosedsurfaces}. In particular, via this identification $U_\Sigma$ is given by the matrix $U$ from~\eqref{eq:BilFormU}.

Assume that \ref{item:KinSigma} holds. By Lemma~\ref{lem:Lpperp=H1(FKppepr)}, $L=H_1(F_1;\Z)\subset H_1(\Sigma;\Z)=\Z^4$ is a symplectic plane with $s_1 = s_{F_1}=[q_L]$, while $s_2 = s_{F_2} = \overline{s_{\overline{F_2}}}=\overline{s_{F_K^\pperp}}=\overline{[q_{L^\pperp}]}$. By Theorem~\ref{thm:nonprimthm},
we have $[q_L] = \overline{[q_{a_1}]}*[q_{a_2}]$, where $(a_1,a_2)=\Phi(L)$, and so
\[
s_1 = [q_L] = \overline{[q_{a_1}]}*[q_{a_2}].
\] Now, Corollary~\ref{cor:nonprimLperpLpperp} gives  $[q_{L^\pperp}]=[q_{a_1}]*[q_{a_2}]$, and so
\[s_2=\overline{[q_{L^\pperp}]}=\overline{[q_{a_1}]}*\overline{[q_{a_2}]}.\]
Since $L$ is symplectic, $q_{a_2}$ is primitive by Lemma~\ref{lem:SympIf1} and $\overline{[q_{a_2}]}=[q_{a_2}]^{-1}\in \mathcal{G}^+_D$. 
In fact, since $L$ is symplectic, $q_{a_2}=ax^2+xy+cy^2$ for some $a,c\in \Z$ with $D=1-4ac$. Hence, by acting with $[q_{a_2}]^{-2}$ on $s_1$,
we find 
\[[ax^2+xy+cy^2]^{-2}*s_1=[q_{a_2}]^{-2}*\left(\overline{[q_{a_1}]}*[q_{a_2}]\right)=\overline{[q_{a_1}]}*\overline{[q_{a_2}]}=s_2,\]
so \ref{item:t^2s1=s2} follows from $[ax^2+xy+cy^2]^{-1} = [cx^2+xy+ay^2]$ and relabeling.

Assume now that \ref{item:t^2s1=s2} holds. 
Set $q_{a_2}=cx^2+xy+ay^2$ and choose $q_{a_1}$ to be an IBQF in the class $[q_{a_2}]^{-1}*\overline{s_2} = [q_{a_2}]*\overline{s_1} \in \mathcal{Q}^+_D$. The corresponding Klein vectors $a_1$ and $a_2$ define a symplectic plane $L$ of $\Z^4=H_1(\Sigma;\Z)$ as in Proposition~\ref{prop:KleinBasicProp}. Let $K\subset \Sigma$ be a separating knot such that $L=H_1(F_K;\Z)\subset H_1(\Sigma;\Z)$, where $F_K$ is the closure of the component of $\Sigma\setminus K$ that induces the correct orientation on $K$; compare Lemma~\ref{lem:realizationofL}.
Since $[q_L]=s_{F_K}$ and
$s_{F_2}=s_{\overline{F_K^\pperp}}=\overline{[q_{L^\pperp}]}$ (by Lemma~\ref{lem:Lpperp=H1(FKppepr)}), we have
\[s_1=\overline{[q_{a_1}]}*[q_{a_2}]=[q_L]=s_{F_1},\]
where the second equation follows from Theorem~\ref{thm:nonprimthm}.
Moreover, using Corollary~\ref{cor:nonprimLperpLpperp}
\[s_2=\overline{[q_{a_1}]}*\overline{[q_{a_2}]}=\overline{[q_{L^\pperp}]}=s_{F_2}\]
as desired.
\end{proof}
Theorem~\ref{thm:main3d} in particular establishes Theorem~\ref{thmintro:main3d} from the introduction.

\begin{proof}[Proof of Theorem~\ref{thmintro:main3d}]
Clearly, Theorem~\ref{thm:main3d} \ref{item:t^2s1=s2}$\Rightarrow$\ref{item:KinSigma} establishes the if-part of Theorem~\ref{thmintro:main3d}.
To see the only-if-part, recall that if $F_1$ and $F_2$ are genus 1 Seifert surfaces for the same knot $K$, then $D_K=-\det(Q_{F_1})=-\det(Q_{F_2})$, where $-\det(Q_{F_i})$ is the discriminant of $s_{F_i}$. Hence setting $D\coloneqq D_K$, the only-if-part is indeed given by Theorem~\ref{thm:main3d} \ref{item:KinSigma}$\Rightarrow$\ref{item:t^2s1=s2}. 
\end{proof}

\subsection{Obstruction to ambient isotopy in $\B^4$ and proof of Theorem~\ref{thmintro:main4d}}\label{ss:obstructionstoisotopyinB4}
Let $F$ be a Seifert surface.
We denote by $F^{\B^4}$ the neatly embedded compact surface in $\B^4$ obtained by pushing $F\setminus \partial F$ into the $4$-ball, i.e.~we take $F^{\B^4}$ to be the time $1$ flow of a smooth vector field on $\B^4$ that is $0$ on $\partial F$ and non-zero radially inward pointing on $S^3\setminus \partial F$.
The surface $F^{\B^4}$ is well-defined up to smooth ambient isotopy relative to boundary.

Surfaces in $\B^4$ arising as the result of pushing in a Seifert surface are in particular smooth submanifolds; however, the obstruction for differentiating two neatly embedded surfaces that we describe in what follows, can be formulated in the topological category. The latter is worthwhile since we hope, in the future, to not only produce obstructions, but potentially use present ideas to provide classification results of pushed-in Seifert surfaces up to topological ambient isotopy. However, in the topological category at many steps one ought to be careful to check things (e.g.~well-definedness of double branched covers of ambiently isotopic surfaces), potentially invoking modern topological 4-manifold theory as e.g.~outlined in \cite{FriedlNagelOrsonPowell}. We avoid considering these topological details, as in this article we only aim to differentiate pushed-in Seifert surfaces, and only need the smooth setup.

For a neatly embedded compact surface $F$ in $\B^4$, we denote by $\Sigma(F)$ the oriented smooth (topological) $4$-manifold resulting from taking the double branched cover of $\B^4$ with branching set $F$. If two neatly embedded compact surfaces $F_1$ and $F_2$ in $\B^4$ can be mapped onto each other by an orientation preserving diffeomorphism (homeomorphism) of $\B^4$, then $\Sigma(F_1)$ and $\Sigma(F_2)$ are orientation preservingly diffeomorphic (homeomorphic). As, if two neatly embedded surfaces $F_1$ and $F_2$ in $\B^4$ are (topologically) \emph{ambiently isotopic}, then, in particular, there exists an orientation preserving diffeomorphism (homeomorphism) $\phi\colon \B^4\to \B^4$ such that $\phi(F_1)=F_2$, we find that the orientation preserving diffeomorphism (homeomorphism) type of the oriented smooth (topological) 4-manifold $\Sigma(F^{\B^4})$ is well-defined.

For an oriented connected compact $4$-manifold $M$ (that may have boundary), we refer to the intersection form $H_2(M;\Z)\times H_2(M;\Z)\to \Z$ on the second integer homology group as \emph{the intersection form of $M$}. For completeness, recall that the intersection form is given by 
\begin{align*}
H_2(M;\Z)\times H_2(M;\Z)\to \Z; (A,B)\mapsto \left(\mathrm{PD}^{-1}(A)\cup \mathrm{PD}^{-1}(B)\right)([M,\partial M]),
\end{align*}
where $[M,\partial M]\in H_4(M,\partial M;\Z)$ denotes the fundamental class and
\begin{align*}
\mathrm{PD}\colon H^2(M,\partial M;\Z)\to H_2(M;\Z)
\end{align*}
denotes Poincar\'e duality.  
The intersection form of $\Sigma(F^{\B^4})$ is isomorphic to $Q_F$, meaning there exists an isometry (i.e.~an isomorphism that preserves the bilinear forms) between $H_2(\Sigma(F^{\B^4});\Z)$ endowed with the intersection form and $H_1(F;\Z)$ endowed with $Q_F$~\cite[Theorem~3]{GordonLitherland_78_OnTheSigOfALink}.

\begin{lem}\label{lem:intersectionformofdoublebranchedcover}
Let $F_1$ and $F_2$ be genus 1 Seifert surfaces for the same knot $K$. Then the following are equivalent:
\begin{enumerate}[(a)]
    \item \label{item:q_Fneqq_F'}
$s_{F_1}$ is different from $s_{F_2}$ and $\overline{s_{F_2}}$,
\item\label{item:cohomologyringdifferent}
$\Sigma(F_1^{\B^4})$ and $\Sigma(F_2^{\B^4})$ have non-isometric intersection forms 
(and, in particular, are not orientation preservingly homeomorphic).
\end{enumerate}

\end{lem}
    The key to Lemma~\ref{lem:intersectionformofdoublebranchedcover} is the observation that,
when $F$ is a genus $1$ Seifert surface of a knot $K$, the isomorphism type of $(H_1(F;\Z), Q_F)$ is determined by the $\mathrm{GL}_2(\Z)$-equivalence class of the IBQF $\frac{Q_F(a,a)}{2}x^2+Q_F(a,b)xy+\frac{Q_F(b,b)}{2}y^2$,
 where $(a,b)$ is any basis for $H_1(F;\Z)$.
\begin{proof}[Proof of Lemma~\ref{lem:intersectionformofdoublebranchedcover}]
Using that the intersection form on $H_2(\Sigma(F_i^{\B^4});\Z)$ is isomorphic to $Q_{F_i}$ on $H_1(F_i;\Z)$, we have that \ref{item:cohomologyringdifferent} holds if and only if there does not exist $M\in \mathrm{GL}_2(\Z)$ such that  $q_{F_1}((x,y)M)=q_{F_2}$.
Here, we have
\begin{align*}
    q_{F_i}&\coloneqq U_{F_i}(a_i,a_i)x^2+(U_{F_i}(a_i,b_i)+U_{F_i}(b_i,a_i))xy +U_{F_i}(b_i,b_i)y^2
    \\&=\frac{1}{2}\left(Q_{F_i}(a_i,a_i)x^2+2Q_{F_i}(a_i,b_i)xy +Q_{F_i}(b_i,b_i)y^2\right)\end{align*}
    for $(a_i,b_i)$ some symplectic basis of $H_1(F_i;\Z)$. 
    The above is equivalent to $s_{F_1}\notin\{s_{F_2},\overline{s_{F_2}}\}$, that is to~\ref{item:q_Fneqq_F'}, as desired.  
\end{proof}
Note that in the following theorem, which is the more precise version of Theorem~\ref{thmintro:main4d}, the case of primitive $s$ corresponds to $s$ with content $1$.
\begin{theorem}\label{thm:main4d} Fix an integer $D \equiv 1 \pmod{4}$. The following are equivalent.
\begin{enumerate}[(a)]
    \item\label{item:nontrivialS+} The subgroup $\mathcal{S}^+_D$ of $\mathcal{G}^+_D$ is non-trivial, i.e.~there exist $a,c\in\Z$ with $1-4ac=D$ such that
    $[ax^2+xy+cy^2]^2\neq 1$.
    \item\label{item:nonequivalentSeifertSurfacesexist} There exist a knot with determinant $D$ and Seifert surfaces $F_1$ and $F_2$ of genus 1 for this knot such that $\Sigma(F_1^{\B^4})$ and $\Sigma(F_2^{\B^4})$ have non-isometric intersection forms (and, in particular, are not orientation preservingly homeomorphic).
\end{enumerate}
Furthermore, for every $m \geq 1$ with $m^2 \mid D$, the following 
are equivalent.
\begin{enumerate}
\item\label{item:1} The image of the subgroup $\mathcal{S}_D^+$ under the map $\mathcal{G}_D^+ \to \mathcal{G}_{D/m^2}^+$ (cf.~Corollary~\ref{cor:modconductor}) has order at least $3$.
\item\label{item:2} For all $s\in \mathcal{Q}_D^+$ with content $m$, there exist Seifert surfaces $F_1$ and $F_2$ of genus 1 for the same knot such that
$s_{F_1}=s$ and $\Sigma(F_1^{\B^4})$ and $\Sigma(F_2^{\B^4})$ have non-isometric intersection forms.
\end{enumerate}
\end{theorem}

In fact, as will be clear from the proof, the Seifert surfaces in~\ref{item:nonequivalentSeifertSurfacesexist} and~\eqref{item:2} can be chosen to have disjoint interiors.

\begin{proof}
To see~\ref{item:nontrivialS+} implies~\ref{item:nonequivalentSeifertSurfacesexist}, assume~\ref{item:nontrivialS+} holds and let $t=[ax^2+xy+cy^2]\in\mathcal{G}^+_D$ with $1-4ac=D$ such that $t^2\neq 1$. We set $s_1=t^2$ and $s_2=1$. By Theorem~\ref{thm:main3d}, there exists a knot $K$ with disjoint Seifert surfaces $F_1$ and $F_2$ such that $s_{F_1}=s_1$ and $s_{F_2}=s_2$; in particular, $K$ has determinant $D$. By Lemma~\ref{lem:intersectionformofdoublebranchedcover}, $\Sigma(F_1^{\B^4})$ and $\Sigma(F_2^{\B^4})$ have non-isometric intersection forms. 

To see~\ref{item:nonequivalentSeifertSurfacesexist} implies~\ref{item:nontrivialS+}, let $F_1$ and $F_2$ be genus 1 Seifert surfaces for the same knot $K$. In case $K$ is the unknot, we have $D_K=1$ and $\mathcal{Q}^+_1$ consists of a single class, thus $s_{F_1}=s_{F_2}$. If $K$ is not the unknot, by~\cite{ScharlemannThompson_88}, there exists a sequence of genus 1 Seifert surfaces $G_1, G_2,\cdots, G_n$ such that $G_i\cap G_{i+1}=K$ for all $1\leq i<n$ and $G_1=F_1$ and $G_n=F_2$. Hence, by applying Theorem~\ref{thm:main3d} to $\Sigma\coloneqq G_i\cup\overline{G_{i+1}}$, we find $t_i=[a_ix^2+xy+c_iy^2]$ such that $t_i^2*s_{G_i}=s_{G_{i+1}}$. Therefore, $(\prod_{i=1}^{n-1} t_i^2)*s_{F_1}=s_{F_2}$. Now assume $F_1$ and $F_2$ are as described in~\ref{item:nonequivalentSeifertSurfacesexist}. In particular, $s_{F_1}\neq s_{F_2}$, hence $t_i^2\neq 1$ for at least one $i$, and by setting $a=a_i$ and $c=c_i$ for this $i$, we have proven~\ref{item:nontrivialS+}.  

To see the last part of the theorem, we note that~\eqref{item:1} is equivalent to the following.
There exist $t_1,t_2\in \mathcal{G}_D^+$ of the form $t_i = [a_i x^2 +xy + c_iy^2]$ where $D = 1-4a_ic_i$ such that the images of $t_1^2$ and $t_2^2$ are distinct and non-trivial under the map $\mathcal{G}_D^+ \to \mathcal{G}_{D/m^2}^+$ (cf.~Corollary~\ref{cor:modconductor}).
To see the forward direction, we note that the set of images of the elements of \eqref{eq:generatingset} generates the image of $\mathcal{S}_D^+$ and this set is closed under inversion, since~\eqref{eq:generatingset} is invariant under $[q]\mapsto[\bar q]=[q]^{-1}$. Thus, if they all were in $\{1,g\}$ for some $g$, then we would have $g=g^{-1}$ and the image of $\mathcal{S}_D^+$ would have order at most $2$.

Assume first that~\eqref{item:1} holds. And let $t_1,t_2 \in \mathcal{G}_D^+$ be as in the formulation of~\eqref{item:1} in the above paragraph.
Set $s_1=s$ and $s_2 = t_1^2*s$ in a first attempt.
Then $s_1 \neq s_2$ (because $\Tilde{t}_1$ is not $2$-torsion) and $s_1 = \overline{s_2}$ if and only if $(\frac{1}{m}s)^2 = \Tilde{t}_1^{-2}$.
If the latter occurs, we replace $t_1$ with $t_2$ and note that $(\frac{1}{m}s)^2 = \Tilde{t}_1^{-2} \neq \Tilde{t}_2^{-2}$ to ensure $s_1 \neq \overline{s_2}$.

By construction, $s_1\notin\{s_2,\overline{s_2}\}$ and $s_2=t_i^2*s_1$ for $i=1$ or $i=2$.
The corresponding Seifert surfaces $F_1$ and $F_2$ as in Theorem~\ref{thm:main3d}\ref{item:KinSigma} are, by Lemma~\ref{lem:intersectionformofdoublebranchedcover}, such that $\Sigma(F_1^{\B^4})$ and $\Sigma(F_2^{\B^4})$ have non-isometric intersection forms as desired.

Assume lastly that Item~\eqref{item:2} holds, and apply it for $s=m\cdot 1_{D/m^2}$. As in the proof of \ref{item:nonequivalentSeifertSurfacesexist}$\Rightarrow$\ref{item:nontrivialS+}, by \cite{ScharlemannThompson_88} and Theorem~\ref{thm:main3d} one finds $u_1\in \mathcal{S}_D^+$ with \[u_1*(m1_{D/m^2})=m(\Tilde u_1*1_{D/m^2})=s_{F_2}\neq s_{F_1}= m\ 1_{D/m^2},\]
where \eqref{eq:reductionComp} is used in the first equality. It follows that $\Tilde u_1\neq 1_{D/m^2}$, and hence there exists $t_1= [ax^2+xy+cy^2]$ with $1-4ac= D$ and $\Tilde t_1^2\neq 1_{D/m^2}$.
We then apply Item~\eqref{item:2} similarly, but now $s=m\Tilde t_1$ (content $m$). 
We find $u_2 \in \mathcal{S}_D^+$ with 
\[s_2:=s_{F_2}=u_2*(m\Tilde t_1)=m(\Tilde u_2*\Tilde t_1)\]
and such that $s_{F_1}\notin\set{s_2,\overline{s_2}}$. Equivalently,
\[
m\Tilde t_1\neq m(\Tilde u_2*\Tilde t_1), \quad m\Tilde t_1\neq 
\overline{s_2}=m(\Tilde u_2^{-1}*\Tilde t_1^{-1}).\]
This yields $\Tilde u_2\neq 1_{D/m^2}$ and $\Tilde u_2\neq \Tilde t_1^{-2}$.
If $\Tilde t_1^{2}=\Tilde t_1^{-2}$, then $1_{D/m^2},\Tilde t_1^2,\Tilde u_2$ are three distinct elements of the image of $\mathcal{S}_D^+$. If $\Tilde t_1^{2}\neq\Tilde t_1^{-2}$, then $1_{D/m^2}$, $\Tilde t_1^2$, $\Tilde t_1^{-2}$ are. Thus, Item~\eqref{item:1} is established.
\end{proof}
This result in particular gives the less technical Theorem~\ref{thmintro:main4d} from the introduction, using that $s\in \mathcal{Q}^+_{D}$ is primitive exactly means that $s$ has content $1$.
\begin{proof}[Proof of Theorem~\ref{thmintro:main4d}]
We note that by Theorem~\ref{thm:main4d}, there exists a knot with determinant $D$ and disjoint Seifert surfaces $F_1$ and $F_2$ such that $\Sigma(F_1^{\B^4})$ and $\Sigma(F_2^{\B^4})$ are not orientation preservingly homeomorphic. 
Hence, there cannot exist an orientation preserving self-homeomorphism $\phi\colon \B^4\to \B^4$ with $\phi(F_1^{\B^4})=F_2^{\B^4}$; i.e.~$F_1^{\B^4}$ and $F_2^{\B^4}$ are not ambiently isotopic.

For the second part of the theorem, assume there exist $t=[ax^2+xy+cy^2]$ and $t'=[a'x^2+xy+c'y^2]$ in $\mathcal{G}^+_{D}$
with $1\neq t^2\neq t'^2\neq 1$. In particular $1,t^2,t'^2$ are distinct elements of $\mathcal{S}^+_D$, which has thus order at least 3. Therefore the statement follows from the furthermore of Theorem~\ref{thm:main4d} with $m=1$:
we find a knot with surfaces $F_1$ and $F_2$ such that $s_{F_1}$ is any fixed $s\in \mathcal{G}^+_{D}$ (i.e.~$s\in \mathcal{Q}^+_{D}$ with content $1$).
    \end{proof}

\subsection{Realization for fixed knots}\label{sec:realizationfixed}

The reader may notice that in our main 3-dimensional results (i.e.~Theorem~\ref{thmintro:main3d} and the more detailed Theorem~\ref{thm:main3d}), we make no claim about the knot that arises. A stronger claim would be the following. Given a genus 1 Seifert surface $F$ for a knot $K$ such that $s_F=s_1$ and $s_2$ with $[ax^2+xy+cy^2]^2*s_1=s_2$ for some $a,c\in\Z$, there exist Seifert surfaces $F_1$ and $F_2$ \emph{for the knot $K$} such that $s_1=s_{F_1}$ and $s_2=s_{F_2}$. 
Or even stronger, one might hope to find those $F_1$ and $F_2$ with the additional restriction that $F_1$ and $F_2$ are disjoint away from the boundary and that $F_1=F$.
In fact, if either of the above were true, in the proof of Corollary~\ref{cor:halfofAlisonIMRN} in \S\ref{ss:Sequiv} below, we could choose all $K_i$ to be the same knot and would obtain a much stronger result, namely that, if $s*s_1=s_2$ for some $s\in\mathcal{S}^+_D$, then $s_1$ and $s_2$ are realized as the Seifert forms of genus 1 Seifert surfaces with boundary the same knot; compare with the conjecture in \S\ref{ss:Sequiv} below and with Problem~\ref{p:S-equi<->realization}.

However, such stronger claims cannot hold in general. In fact, even for the case of $D=-23$, there is an example of a knot for which such a result cannot be obtained.

\begin{example}
The knot $K=9_5$ has exactly two non-isotopic Seifert surfaces $F$ and $F'$ of genus 1. These surfaces have associated Seifert forms $[3x^2-xy+2y^2]$ and $[3x^2+xy+2y^2]$, respectively. (In fact, $F$ and $F'$ become smoothly ambiently isotopic relative boundary when pushed into the $4$-ball.) We elaborate. To see two distinct surfaces, we take $F$ and $F'$ as the result obtained by plumbing together a $3$-times twisted unknotted band to a $2$-times twisted unknotted band as indicated in Figure~\ref{fig:9_5}.
\begin{figure}[ht]
  \centering
  \def\svgscale{0.8}
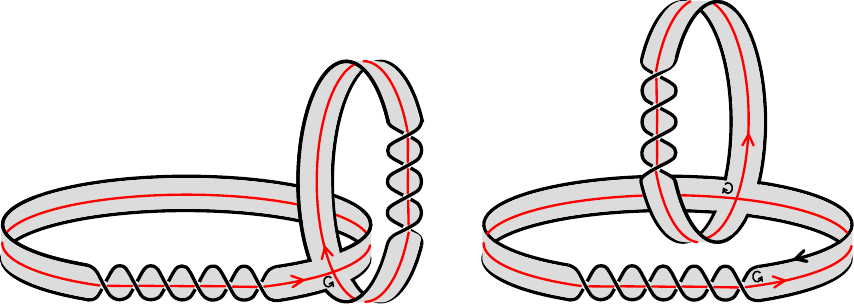
  \caption{The two Seifert surfaces (gray) for $K$. The simple closed curves $\alpha$ and $\beta$ (red) are such that $([\alpha], [\beta])$ forms a symplectic basis of the first homology of the surfaces.}\label{fig:9_5}
\end{figure}They are readily seen to have Seifert forms given by $\big(\begin{smallmatrix}
  2 & 1\\
  0 & 3
\end{smallmatrix}\big)$ and $\big(\begin{smallmatrix}
  2 & 0\\
  -1 & 3
\end{smallmatrix}\big)$ with respect to the basis $[\alpha]$ and $[\beta]$ of first homology indicated in Figure~\ref{fig:9_5}. The Seifert surfaces are non-isotopic in $S^3$ since  $[2x^2+xy+3y^2]\neq [2x^2-xy+3y^2]\in \mathcal{G}_{-23}^+$. In fact, the surfaces become isotopic if the orientation on one of them is reversed, which is reflected in the fact that $[2x^2+xy+3y^2]=\overline{[2x^2-xy+3y^2]}$.
There are no other genus 1 Seifert surfaces up to isotopy by~\cite[Theorem~A]{Kakimizu_05}.

As a result, we have a knot $K$ with a genus 1 Seifert surface $F$, such that there does not exist a Seifert surface $F_2$ of genus 1 for $K$ with $s_{F_2}=1\in \mathcal{G}_{-23}^+$, even though there are $a$ and $c$ such that $[ax^2+xy+cy^2]^2*s_1=s_2$, where $s_1=s_F=[3x^2-xy+2y^2]$ and $s_2=1$. 
\end{example}

\subsection{S-equivalence}\label{ss:Sequiv}

This section is written for the knot theorist who is used to seeing Seifert matrices in the context of S-equivalence and would like to know how our algebraic setup fits in with that.  It is independent from the rest of the paper.

We comment further on the condition in Theorem~\ref{thmintro:main3d} respectively Theorem~\ref{thm:main3d}, which states that IBQFs $s_1$ and $s_2$ can be realized as the Seifert forms of a disjoint genus $1$ pair if and only if there exists a class $t=[ax^2+xy+cy^2]$ such that $t^2*s_1=s_2$.


After discovering Theorem~\ref{thmintro:main3d} through an algebraic calculation, the authors realized that the necessary and sufficient condition fits with the following prior result by the third author. 
For a fixed discriminant $D\equiv 1 \pmod{4}$, recall that we denote by $\mathcal{S}^+_D$ the subgroup of $\mathcal{G}^+_D$ generated by elements of the form $[ax^2+xy+cy^2]^2$. The third author has described and investigated the notion of S-equivalence for Seifert forms of Seifert surfaces of genus 1 for knots using this subgroup~\cite[\S2]{Miller_22}.  In doing so, she recovered some results that were found much earlier by Dean Bandes in his  Ph.D. thesis \cite{bandes-thesis}.
In the notation of the present paper, the relevant result is the following.
\begin{theorem}[{\cite[Thm.~14.1]{bandes-thesis}, \cite[Prop.~2.9]{Miller_22}}]\label{thm:AlisonIMRN}
    Let $s_1$ and $s_2$ be in $\mathcal{Q}^+_D$. Then $s_1$ and $s_2$ are S-equivalent if and only if there exists an element $s\in\mathcal{S}^+_D$ such that $s*s_1=s_2$.
In other words, $s_1$ and $s_2$ are S-equivalent if and only if there exists a finite number of $t_1,\dots,t_n$ of the form $t_i=[a_ix^2+xy+c_iy^2]$ such that $(\prod_{i=1}^{n}t_i^2)*s_1=s_2$. 
\end{theorem} 
Theorem~\ref{thm:AlisonIMRN} is a translation of \cite[Prop.~2.9]{Miller_22} into the language of binary quadratic forms via the correspondence described in \cite[\S2]{Miller_22}.  Alternatively, it is a translation of \cite[Thm.~14.1]{bandes-thesis} into the language of binary quadratic forms via the correspondence described in \cite[Thm.~9.1]{bandes-thesis}.  
(The strict similarity classes of modules defined in \cite{bandes-thesis} are essentially narrow ideal classes.)

While we abstain from defining the notion of S-equivalence beyond giving the above equivalent formulation, we note that the notion is of interest and arises naturally in knot theory due to the following.

\begin{lem}[{\cite[Theorem~8.4]{Lickorish_97}}]\label{lem:sameboundary=>S-equiv}
    All Seifert forms that arise from Seifert surfaces with boundary a fixed knot are S-equivalent.
\end{lem}

To clarify the connection between our result and the earlier work of the third author, we describe how to use Theorem~\ref{thmintro:main3d} respectively Theorem~\ref{thm:main3d} to recover the following statement, which is the easier half of Theorem~\ref{thm:AlisonIMRN}.
\begin{cor}\label{cor:halfofAlisonIMRN}
Fix $D\equiv 1 \pmod{4}$. Let $s_1$ and $s_2$ be in $\mathcal{Q}^+_D$. If there exists an $s\in\mathcal{S}^+_D$ such that $s*s_1=s_2$, then $s_1$ and $s_2$ are S-equivalent.
\end{cor}
\begin{proof} By the definition of $\mathcal{S}^+_D$, there exist $t_1,\dots t_n$ of the form $t_i=[a_ix^2+xy+c_iy^2]$ such that $s=\prod_{i=1}^{n}t_i^2$. By Theorem~\ref{thm:main3d}, for all $1\leq k\leq n$, there exists a knot $K_k$ with a pair of genus $1$ Seifert surfaces $F_k$ and $F_k'$ (that are disjoint in their interior) such that $\prod_{i=1}^{k-1}t_i^2*s_1=s_{F_k}$ and $\prod_{i=1}^{k}t_i^2*s_1=s_{F_k'}$. In particular, for all $1\leq k\leq n$, $\prod_{i=1}^{k-1}t_i^2*s_1=s_{F_k}$ and $\prod_{i=1}^{k}t_i^2*s_1=s_{F_k'}$ are S-equivalent by Lemma~\ref{lem:sameboundary=>S-equiv}. Since S-equivalence is an equivalence relation and as such is transitive, we find $s_1$ and $s_2=\prod_{i=1}^{n}t_i^2*s_1$ are S-equivalent.
\end{proof}

Note that in the above proof, we did \emph{not} establish that $s_1$ and $s_2$ can be realized as the Seifert forms of Seifert surfaces for the \emph{same} knot.
However, given Theorem~\ref{thm:main3d}, which for the special case of $s*s_1=s_2$ with $s=[ax^2+xy+cy^2]^2$ not only establishes S-equivalence but in fact realization as Seifert forms for Seifert surfaces for the same knot, we conjecture the following.
For all $D\equiv 1 \pmod{4}$, if $s\in\mathcal{S}^+_D$ and $s_1,s_2\in\mathcal{Q}^+_D$ satisfy $s*s_1=s_2$, then there exist genus $1$ Seifert surfaces $F_1$ and $F_2$ with boundary the same knot such that $s_1=s_{F_1}$ and $s_2=s_{F_2}$; see Problem~\ref{p:S-equi<->realization}.

\begin{remark}\label{rem:altproof} We provide a proof of~\ref{item:nonequivalentSeifertSurfacesexist} implies~\ref{item:nontrivialS+} of Theorem~\ref{thm:main4d} using Theorem~\ref{thm:AlisonIMRN} in place of making use of~\cite{ScharlemannThompson_88}.
Assume $F_1$ and $F_2$ are Seifert surfaces as described in~\ref{item:nonequivalentSeifertSurfacesexist}. By Lemma~\ref{lem:sameboundary=>S-equiv}, $s_{F_1}$ and $s_{F_2}$ are S-equivalent; hence, by Theorem~\ref{thm:AlisonIMRN}, there exist
$t_1,\dots,t_n$ of the form $t_i=[a_ix^2+xy+c_iy^2]$ such that $(\prod_{i=1}^{n}t_i^2)*s_{F_1}=s_{F_2}$. Now, as in the proof of Theorem~\ref{thm:main4d}, since $s_{F_1}\neq s_{F_2}$, $t_i^2\neq 1$ for at least one $i$.  
\end{remark}

We end this subsection with more philosophical comments. Firstly, in hindsight, in light of the third author's earlier results, the condition we found in Theorem~\ref{thmintro:main3d} and Theorem~\ref{thm:main3d}, respectively, make sense, as they are a refinement of what needs to hold for S-equivalent $s_1$ and $s_2$. However, this is not how we discovered them. 
They arose by direct calculation using Theorem~\ref{thm:nonprimthm}, we only later noticed the connection. Secondly, the strength of Theorem~\ref{thm:main3d} (certainly with view of the application to pushed-in Seifert surfaces) lies in the fact the IBQFs are realized by Seifert surfaces with the \emph{same} knot as its boundary. Instead, the third author's prior result is an algebraic statement and does not give realization for knots (however it does imply a realization result in the analogous high-dimensional setting of simple $(4a+1)$-knots, where two knots with S-equivalent Seifert pairings are necessarily isotopic).
The natural analog of the third author's result (Theorem~\ref{thm:AlisonIMRN}) would be the positive resolution of the conjecture made in the above paragraph. As an immediate application, one could find many further examples of forms that are realized by Seifert surfaces with the same boundary and hence could use that to provide many more examples of Seifert surfaces that are non-isotopic even when pushed into the $4$-ball.

\section{Examples}\label{sec:examples}

In this section we prove Corollary~\ref{corintro} and
give a variety of examples to illustrate the main results. The explicit calculations in the examples below can e.g.~be done using standard descriptions of Gauss composition~\cite{gauss}, \cite[Chap.~14]{cassels}. 
We implemented and verified most of them using SageMath.
    
\subsection{Negative discriminants}\label{sec:negdisc}

We begin the discussion with proving Corollary~\ref{corintro}. 
The following is the main observation.

\begin{lem}\label{lem:quadnottriv-neg}
    Suppose that $D$ is a negative discriminant with $D \equiv 1 \mod 4$.
    Then $[ax^2+xy+cy^2]^2 = 1$ for all $a,c\in \Z$ with $D = 1-4ac$ if and only if $D$ is of the form $D = 1-4p^k$ for $ k \in \{0,1,2\}$ and $p$ a prime.
\end{lem}

\begin{proof}
We first claim that the square of any class $[ax^2+xy+cy^2]$ is given by
\begin{align}\label{eq:squareclass}
    [ax^2+xy+cy^2]^2 = [a^2x^2+(1-2ac)xy+c^2y^2].
\end{align}
This can be checked directly by Gauss composition.
Alternatively, and in the spirit of the current article, let us recall that the square class  $[ax^2+xy+cy^2]^2$ is exactly minus the class of the quadratic form on the oriented plane $\langle 1,A\rangle^\perp$ where $A = \SmMat{1}{-2a}{2c}{-1}$ (cf.~Corollary~\ref{cor:nonprimLperpLpperp}). 
An oriented (integer) basis of $\langle 1,A\rangle^\perp = L_{A,-A}$ is easily computed to be $\SmMat{a}{0}{1}{-a}$ and $\SmMat{-c}{1}{0}{c}$ which implies \eqref{eq:squareclass}.

We may assume $a,c>0$.
As the trivial class $x^2+xy+ac y^2$ primitively represents no integers less than $ac$ except for $1$, the square class \eqref{eq:squareclass} is non-trivial unless $a=1$, $c=1$, or $a=c$ (in which case it is trivial).
Turning to the statement of Lemma~\ref{lem:quadnottriv-neg}, $D$ may be written as $D = 1-4ac$ for $a,c>1$ with $a \neq c$ if and only if it is not of the form $D = 1-4p^k$ for $ k \in \{0,1,2\}$ and $p$ a prime. In this case exactly, we can find a non-trivial class as in \eqref{eq:squareclass} and hence the lemma follows.
\end{proof}

 \begin{proof}[Proof of Corollary~\ref{corintro}]
This follows directly from Lemma~\ref{lem:quadnottriv-neg} and Theorem~\ref{thm:main4d}.
 \end{proof}

In the following example we discuss the discriminant $D = -23$ in full detail. In particular, this connects to \cite{HMKPS22} and extends upon Example~\ref{exp:HMKPS1}.

\begin{example}\label{ex:-23}
The oriented class group for discriminant $-23$ has representatives
\begin{align*}
\mathcal{G}_D^+ = \{[x^2+xy+6y^2],[-x^2-xy-6y^2],[2x^2\pm xy+3y^2],[-2x^2\pm xy-3y^2]\}.
\end{align*}
As a group of order $6$, it is isomorphic to $\Z/2\Z \times \Z/3\Z$.
Recall from Example~\ref{exp:HMKPS1} that by Theorem~\ref{thmintro:main3d}, two Seifert forms $s_1, s_2\in\mathcal{G}^+_{-23}$ arise from a disjoint genus 1 pair of Seifert surfaces if and only if
$[2x^2\pm xy +3y^2]*s_1=s_2$ or $s_1=s_2$.
The unordered pairs of such $s_1$ and $s_2$ are, in addition to the six with $s_1=s_2$, the following:
\begin{align*}
\{1,[2x^2 + xy +3y^2]\},\{1,[2x^2 -xy +3y^2]\},
\{[-x^2-xy-6y^2],[-2x^2-xy -3y^2]\},\\
\{[-x^2-xy-6y^2],[-2x^2 +xy -3y^2]\},
\{[2x^2 +xy +3y^2],[2x^2 -xy +3y^2]\},\\\text{ and }\quad \{[-2x^2 -xy -3y^2],[-2x^2 +xy -3y^2]\}.
\end{align*}
For a pair of Seifert surfaces of the same knot that realize $s_1$ and $s_2$, we have by Lemma~\ref{lem:intersectionformofdoublebranchedcover} that the surfaces are not ambient isotopic in $\B^4$ whenever $s_1^{-1}=\overline{s_1}\neq s_2 \neq s_1$. Of the above pairs with $s_1\neq s_2$, only the last two satisfy $s_1^{-1}=s_2$. Hence, we have that the other $4$ pairs arise from Seifert surfaces as asked for by Livingston. 
The example chosen in \cite[Proof of Theorem 4.1]{HMKPS22} is the fourth of the above pairs: $[-x^2-xy-6y^2]$ and $[-2x^2 +xy -3y^2]=[-6x^2-5xy-2y^2]$. The other three pairs are obtained from this one by taking mirror image (switches from negative definite to positive definite) and reversing orientation of the Seifert surfaces (corresponds to taking inverses).
\end{example}

Observe that in the above example a pair $s_1,s_2$ arises from a disjoint genus one pair of Seifert surfaces if and only if there exists $s \in \mathcal{S}_D^+$ with $s *s_1 = s_2$ (in the setting of \S\ref{ss:Sequiv} this means that $s_1,s_2$ are S-equivalent).
This is not generally true, as the example of $D=-71$ shows.

\begin{example}\label{exp:-71}
Consider $D = -71$ and note that $-71 = 1-4\cdot 2 \cdot 9$. 
The non-trivial classes of discriminant $-71$ of the form $[ax^2+xy+cy^2]^2$, $ac = 18$, are
$[3x^2\pm xy + 6y^2]^2 = [2x^2\mp xy+9y^2]$ and $[2x^2\pm xy + 9y^2]^2 = [4x^2\mp 3xy+5y^2]$.
The oriented class group $\mathcal{G}_{-71}^+$ is a group of order $14$ where the (non-oriented) class group $\mathcal{G}_D\cong \Z/7\Z$ is
\begin{align*}
\mathcal{G}_D = \{[x^2 + xy + 18y^2], [2x^2 \pm xy + 9y^2],[3x^2\pm xy+6y^2], [4x^2\pm 3xy+5y^2]\}.
\end{align*}
The cyclic group $\Z/7\Z$ is generated by any non-trivial element and hence $\mathcal{S}_{-71}^+ = (\mathcal{G}_{-71}^+)^2$.
On the other hand, it is clear that (for example) the classes $s_1= [x^2+xy+18y^2],s_2 = [3x^2+xy+6y^2]$ are not related by multiplication with a class of the form $[ax^2+xy+cy^2]^2$.
In particular, our results do not answer whether or not there exist a knot $K$ and two genus one Seifert surfaces $F_1,F_2$ with boundary $K$ such that $s_{F_1}= s_1$ and $s_{F_2} = s_2$.
However, Theorem~\ref{thmintro:main3d} shows that such $F_1,F_2$, if they exist, cannot be disjoint in their interior.
We also refer to Problem \ref{p:S-equi<->realization} for a general formulation.
\end{example}

\begin{example}\label{exp:Feher}
Feh\'er \cite{Feher} constructs an infinite set of examples of genus one disjoint pairs of Seifert surfaces and shows that they are not ambiently isotopic in the $4$-ball.
In the following, we discuss their work in the context of the current article.
Let $p,q>1$ be coprime integers, $n \geq 1$ (which will later be chosen large), and $k \in \Z$.
Choose $r,s \in \Z$ with $ps-qr=1$.
Then there exists a symplectic plane $L$ such that
\begin{align*}
[q_L] &= [pqx^2+(1-2kp)xy+ny^2],\\
[q_{L^\pperp}] &= [pqx^2+(2kp-1-2qr)xy+(rs-2kr+n)y^2].
\end{align*}
Feh\'er establishes this by drawing disjoint genus one pairs of Seifert surfaces realizing these forms.
In the context of this article, one may verify that the plane $L = L_{a_1,a_2}$ for
\begin{align*}
a_1 = \begin{pmatrix}
-1+2kp & 2q \\ -2np & 1-2kp
\end{pmatrix},\
a_2 = \begin{pmatrix}
1 & 2p \\ 2(k^2p-k-nq) & -1
\end{pmatrix}.
\end{align*}
gives rise to the desired forms.

Feh\'er \cite{Feher} verifies by direct elementary means that $[q_L] \neq [q_{L^\pperp}],[\overline{q_{L^\pperp}}]$ whenever $2kp\not\equiv 1 \mod q$ and
\begin{align*}
n \geq \frac{k(kp-1)}{q} + \frac{pq}{12}-\frac{1}{6}+\frac{1}{2pq}.
\end{align*} 
In particular, this shows that the corresponding disjoint genus one pairs of Seifert surfaces remain non-isotopic when pushed into the $4$-ball.
Corollary~\ref{cor:stabilizer} implies that $[q_L] \neq [q_{L^\pperp}],[\overline{q_{L^\pperp}}]$ if and only if 
\begin{align}\label{eq:ourFehercond}
[q_{a_1}] \neq [\overline{q_{a_1}}],\quad [q_{a_2}]^2 \not\in \mathrm{Stab}([q_{a_1}]).
\end{align}
As in Lemma~\ref{lem:quadnottriv-neg} one may verify that this is, for instance, the case if $2kp\not\equiv 1 \mod q$, $\gcd(1-2kp,q,n)=1$, and
\begin{align*}
n > \frac{k(kp-1)}{q}+ \max\Big\{\frac{p}{q},\frac{q}{p}\Big\}+ \frac{1}{4pq}.
\end{align*}
Moreover, for specific values of $p,q,k,n$ the conditions in \eqref{eq:ourFehercond} can be quickly verified using reduction theory.
\end{example}

\subsection{Positive discriminants}

Generally, determining class groups for positive discriminants is (much) harder than for negative discriminants (see e.g.~Gauss's class number problem).

\begin{example}
The discriminant $D = 145$ is the smallest positive squarefree discriminant for which there exists a non-trivial special class $[ax^2 \pm xy + c y^2]^2$ with $ac = -(D-1)/4 = -36$.
Such a non-trivial class is given by $[2x^2+xy-18y^2]^2 = [6x^2+5xy-5y^2]$.
We note that the oriented class group in this case is isomorphic to $\Z/4\Z$ with generator $[3x^2+7xy-8y^2]$.
\end{example}

\begin{example}
The discriminant $D = 905$ is a positive and squarefree discriminant for which there exists a class $s \in (\mathcal{G}_D^+)^2$ which is not of the form $[ax^2\pm xy+cy^2]^2$ for any $a,c$ with $1-4ac = D$.
In fact, a calculation verifies that $\mathcal{G}_D^+ \simeq \Z/8\Z$ so that $(\mathcal{G}_D^+)^2 \simeq \Z/4\Z$.
As $(905-1)/4 = 226 = 2\cdot 113$, there exist at most $3$ classes of the form $[ax^2 \pm xy + c y^2]^2$ which verifies the above claim.
This situation is analogous to Example~\ref{exp:-71}.
\end{example}

\subsection{Square discriminants}\label{sec:squaredisc}
\subsubsection{Reduction theory and Gauss composition for square discriminants}

Many treatments of the theory of binary quadratic forms exclude the case where the discriminant $D$ is a square number.  
Although this case is in some regards arithmetically less interesting than the others, as observed in \cite{bhargava_composition}, Gauss composition still makes sense in this context.
In fact, as we will see below, the group $\mathcal{G}^+_{D}$ admits an explicit description when $D = N^2$:  it is isomorphic to the multiplicative group modulo $N$.

From a topological viewpoint, the case $D = N^2$ is precisely that of genus one algebraically slice knots, which we now define (see \cite{casson-gordon-slice} for a reference).

\begin{definition}
    A knot $K$ is {\em algebraically slice} if for any Seifert surface $F$ of $K$ with genus $g$, there is a rank $g$ subgroup $A$ of $H_1(F; \Z) \cong \Z^{2g}$ on which the restriction of the Seifert form $U_F$ is zero.
\end{definition}
 
The defining property of algebraically slice knots is known to depend only on the knot $K$ and not on the choice of surface $F$.  That is, it suffices to check this property for a single surface $F$ to show that a knot is algebraically slice.  The terminology ``algebraically slice'' comes from the fact that topologically slice knots (knots which arise as $K = \partial \Delta \subset \partial D^4 = S^3$ for a topologically embedded disk $\Delta \subset D^4$ in the 4-ball) are known to be algebraically slice, although the converse is not true.  

The following proposition characterizes genus 1 algebraically slice knots as those whose associated quadratic form has square discriminant.  (This is well-known to knot theorists but we write it out here in detail for the sake of completeness.)

\begin{prop}
    For a knot $K$ with a genus one Seifert surface $F$, the following are equivalent.
    \begin{enumerate}[(i)]
        \item $K$ is algebraically slice.
        \item There exists $0\neq v \in H_1(F; \Z)$ with $U_F(v, v) = 0$.
        \item There exists $0\neq v \in H_1(F; \Z)$ with $Q_F(v, v) = 0$.
        \item The IBQF $s_F$ defined in \eqref{eq:s_F} takes the form $s_F = [0 x^2 +  b xy +  c y^2]$.
        \item The discriminant of the IBQF $s_F$ is a square.
        \item The determinant $D_K$ of the knot $K$ is a square. (Recall that we have chosen the sign convention so that $D_K$ is always $1 \bmod 4$.)
    \end{enumerate}
\end{prop}

\begin{remark}
For readers familiar with the Alexander polynomial, we note one more statement, whose equivalence with (vi) follows immediately from the formula for the Alexander polynomial of a knot of genus 1 in terms of its determinant: (vii) The Alexander polynomial of $K$ factors over $\Z$ as $\Delta(t) = (kt - (k+1))(kt^{-1} -(k+1))$ for some integer $k$.  
\end{remark}

\begin{proof}
    For (i) $\implies$ (ii), take $v$ to be a generator of the subspace $A \subset H_1(F; \Z)$.  Conversely, for (ii) $\implies$ (i), take $A$ to be the subspace spanned by $v$.

    The equivalence of (ii) and (iii) is immediate from the equality $Q_F(v, v) = 2 U_F(v, v)$.

    For (iii) $\implies$ (iv), replace $v$ by the primitive element given by $v$ divided by a positive integer, and extend it (denoted $v$ by abuse of notation) to a symplectic basis $(v, w)$ of $H_1(F;  \Z)$ and use \eqref{eq:s_F}.  Conversely, if $s_F$ is $\SL_2$-equivalent to $0 x^2 +  b xy +  c y^2$ we can change to a symplectic basis $(v, w)$ in which
    \[
    U_F(v,v)x^2+(U_F(v,w)+U_F(w,v))xy +U_F(w,w)y^2 = 0 x^2 + bxy + cy^2
    \]
    and so in particular $U_F(v, v) = 0$.

    (iv) implies (v) since $[0x^2 + bxy + cy^2]$ has discriminant $b^2$. After applying the change of variables $(x, y) \mapsto (y, -x)$ to move the $0$ coefficient to the $y^2$ term, the converse implication becomes precisely Proposition~\ref{prop:reductionsquare}, and we defer the proof to there.  

    (v) is equivalent to (vi) since $D_K = \disc s_F$.
\end{proof}

For lack of (explicit) reference, we now write down the details of reduction theory and Gauss composition for the case of square discriminants $D = N^2$.

\begin{prop}[Reduction Theory for Square Discriminant]\label{prop:reductionsquare}
Any quadratic form of discriminant $D = N^2$ is $\SL_2(\Z)$-equivalent to a unique one of the forms
\[
Q_{N, a} = ax^2 + Nxy + 0 y^2.
\]
as $a$ runs through a complete set of representatives of the residue classes modulo $N$.
\end{prop}
\begin{proof}
Since any quadratic form is a positive scalar multiple of a primitive form, it suffices to show this for primitive quadratic forms.

Let $Q$ be a primitive quadratic form of discriminant $D = N^2$. 
Since the discriminant is square, $Q$ factors over the rationals, and so also over the integers by Gauss's lemma.
Write $Q(x, y) = (\alpha x + \gamma y) (\beta x + \delta y)$ for integers $\alpha, \beta, \gamma, \delta$.
One verifies that for $g \in \SL_2(\Z)$ we have 
\begin{equation}\label{eq:SL2 action on factors}
g.Q(x,y) = (\alpha' x + \gamma' y)(\beta' x + \delta' y)
\end{equation}
where $\left(\begin{smallmatrix}\alpha' & \beta' \\ \gamma' & \delta'\end{smallmatrix} \right) = g \cdot \left(\begin{smallmatrix}\alpha & \beta \\ \gamma & \delta\end{smallmatrix} \right) $.
One can compute $N^2 = \disc Q = \det \left(\begin{smallmatrix}\alpha & \beta \\ \gamma & \delta\end{smallmatrix} \right)^2$. 
Swapping the columns of the matrix flips the sign of the determinant, so we can choose our labeling of the factors so that $\det \left(\begin{smallmatrix}\alpha & \beta \\ \gamma & \delta\end{smallmatrix} \right) = N$.

By Gauss's lemma, primitivity of $Q$ is equivalent to primitivity of the vectors $(\alpha, \gamma)$ and $(\beta, \delta)$.
Thus, we may choose a matrix $g \in \SL_2(\Z)$ such that $g \left(\begin{smallmatrix} \alpha  \\ \gamma\end{smallmatrix} \right) = \left(\begin{smallmatrix} 1  \\ 0\end{smallmatrix} \right)$.  
By the above
\begin{equation}
    g \begin{pmatrix}\alpha & \beta \\ \gamma & \delta\end{pmatrix} = \begin{pmatrix} 1 & a\\ 0 & N\end{pmatrix}
\end{equation}
It follows from \eqref{eq:SL2 action on factors} that 
\begin{equation}
g.Q(x, y) = x (a x + Ny) = ax^2 + Nxy + 0y^2.
\end{equation}
Note however that $g$ is only well-defined up to left multiplication by the stabilizer $\left(\begin{smallmatrix} 1 & k \\ 0 & 1 \end{smallmatrix}\right)$ of the vector $\left(\begin{smallmatrix} 1  \\ 0\end{smallmatrix} \right)$.  This action sends $ax^2 + Nxy + 0y^2$ to $(a + Nk) x^2 + Nxy + 0y^2$, so $a$ is only well-defined modulo $N$.
We can choose $g$ such that $a$ lands in a given set of representatives of the primitive residue classes modulo $N$.

To obtain uniqueness, let $0 \leq a,a' < N$ be representatives of primitive residue classes modulo $N$ and suppose that $g \in \SL_2(\Z)$ is such that $g.Q_{N,a} = Q_{N,a'}$. 
By \eqref{eq:SL2 action on factors}, and after swapping $g$ to $-g$ if needed, we have either
\begin{align*}
g \begin{pmatrix}1 & a \\ 0 & N\end{pmatrix} = \begin{pmatrix} 1 & a'\\ 0 & N\end{pmatrix}
\text{ or }
g \begin{pmatrix}1 & a \\ 0 & N\end{pmatrix} = \begin{pmatrix} a' & 1\\ N & 0\end{pmatrix}.
\end{align*}
In the former case, this implies $a=a'$. In the latter case, 
$g = \SmMat{a'}{(-aa'+1)/N}{N}{-a}$ which has determinant $-1$; hence this case does not occur.
This proves the proposition.\end{proof}

\begin{prop}[Gauss Composition for Square Discriminant]\label{prop:Gausscompsquare}
When $D = N^2$ is a square discriminant, the group of $\SL_2(\Z)$-equivalence classes of primitive binary quadratic forms under Gauss composition is canonically isomorphic to the multiplicative group modulo $N$ via the following map
\begin{equation}
\phi_N:  (\Z/N \Z)^\times \overset{\sim}{\longrightarrow} \mathcal{G}^+_{N^2}, \quad a \bmod N \mapsto [Q_{N, a}]=[ax^2 + N xy + 0 y^2]. 
\end{equation}
\end{prop}
\begin{proof}
    The map $\phi_N$ is well-defined because $[ax^2 + Nxy]$ is $\SL_2(\Z)$-equivalent to $[(a+kN)x^2 + Nxy]$ for all integers $k$.
    Proposition~\ref{prop:reductionsquare} then says precisely that $\phi_N$ is a bijection.  
    Applying the Dirichlet composition law to two concordant forms $a_1 x^2 + Nxy + 0 y^2$ and $a_2 x^2 + Nxy + 0y^2$ (cf.~\cite[p.~335]{cassels}) with $a_1$ and $a_2$ non-zero and coprime, we see that $\phi_N(a_1 \bmod N) \phi_N(a_2 \bmod N) = \phi_N(a_1 a_2 \bmod N)$, so $\phi_N$ is also a group homomorphism, hence an isomorphism.
\end{proof}

\subsubsection{Computations and examples}

In view of Proposition~\ref{prop:Gausscompsquare} it is remarkably simple to compute the set of special classes $[ax^2 + xy + cy^2]$ for odd square discriminants.

\begin{lem}\label{lem:squaredisc-specialclasses}
Suppose that $D = N^2$ is an odd square discriminant and write $N = 2k+1$, $k \geq 1$. Then 
\begin{align*}
\{[ax^2 + xy + cy^2]&: 1-4ac = D \} = \phi_N(\{ 2 \delta \beta \mod N : \delta, \beta \in \Z, \beta \mid k, \delta \mid k+1 \}).
\end{align*}
\end{lem}

\begin{proof}
The quadratic form $ax^2 +xy + cy^2$ has discriminant $(2k+1)^2$ if and only if $ac = -k(k+1)$.
As above, by Gauss's lemma there exist integers $\alpha,\beta,\gamma,\delta$ such that 
\[
ax^2 + xy + cy^2 = (\alpha x + \gamma y) (\beta x + \delta y).
\]
Thus, $\alpha\beta =a$, $\gamma\delta=c$, and $\beta\gamma+\alpha\delta=1$. 
Therefore, $\alpha\beta\gamma\delta=-k(k+1)$ and, up to switching the factors, we can assume $\alpha\delta=k+1$ and $\beta\gamma=-k$.
Let $g$ be the matrix $\left( \smallmatrix \delta & \beta  \\ -\gamma &\alpha \endsmallmatrix \right)$; note that $g \in \SL_2(\Z)$.  Then
\begin{equation}
\begin{split}
g.(ax^2 + xy + cy^2) &= (\alpha (\delta x - \gamma  y) + \gamma (\beta x + \alpha y))
(\beta (\delta x - \gamma y) + \delta (\beta x + \alpha y)) \\
&= x (2 \delta \beta x + (2k+1) y) \\
&= Q_{N, 2 \delta \beta} 
\end{split}
\end{equation}
since $N = 2k+1$.
Conversely, for $\beta \mid k$ and $\delta\mid (k+1)$ the form $Q_{N, 2 \delta \beta}$ is equivalent to $\frac{k+1}{\delta}\beta x^2 + xy - \frac{k}{\beta}\delta y^2$.
\end{proof}

\begin{cor}\label{cor:nontrivclass-square}
If $N > 3$ is an odd number, then there exists a knot with determinant $N^2$ having two Seifert surfaces that are not ambiently isotopic in $\B^4$.
\end{cor}
\begin{proof}
Take $\beta = \delta = 1$ in Lemma~\ref{lem:squaredisc-specialclasses} so that the class $[q] = \phi_N(2 \bmod N)$ is of the form $[ax^2+xy+cy^2]$ (as the proof of the lemma shows, that class is $[(k+1)x^2 + xy - k y^2]$).
Since $\phi_N$ is injective, the class $[q]^2 = \phi_N(4 \bmod N)$ is not the identity when $N > 3$.  
We then apply Theorem~\ref{thmintro:main4d} to obtain the desired result. 
\end{proof}
\section{Open Problems}\label{sec:openproblems}

\subsection{Purely number-theoretic problems}\label{subsec:problemsnumertheory}
As previously noted, analyzing the class group for positive discriminants is much harder than for negative discriminants (see e.g.~Gauss's class number problem).
This is also illustrated in the strength of Corollary~\ref{corintro} (see also Lemma~\ref{lem:quadnottriv-neg}).

\begin{problem}
Characterize the positive discriminants $D$ for which a class $[ax^2+xy+cy^2]$ for $a,c\in \Z$, $D = 1-4ac$, exists with $[ax^2+xy+cy^2]^2 \neq 1$.
\end{problem}

If $D$ is a square, this is addressed in Corollary~\ref{cor:nontrivclass-square}.
For fundamental discriminants (i.e. $D$ square-free and $D \equiv 1 \mod 4$), the following problem is a weakening of the above.

\begin{problem}
Show that there are infinitely many positive fundamental discriminants $D$, $D \equiv 1 \mod 4$, for which a class $[ax^2+xy+cy^2]$ with $a,c\in \Z$, $D = 1-4ac$, exists such that $[ax^2+xy+cy^2]^2 \neq 1$. 
If there are infinitely many such discriminants, establish or estimate the density of the set of such discriminants. 
\end{problem}

The problems addressed in this article are easily posed for higher genera as follows.
Let $U$ be the bilinear form on $\Z^{4g}$ given by $U(e_i,e_j) =1$ if $(i,j) \in \{(1,2),(3,4),\ldots\}$ and $0$ otherwise.
Let $\theta$ be the antisymmetrization of $U$, and, for a direct summand $L$ of $\Z^{4g}$, denote the restriction of the quadratic form $x\mapsto U(x,x)$ to $L$ by $q_L$.
We call a rank $2g$ direct summand $L$ of $\Z^{4g}$ symplectic if there exists a basis $v_1,\ldots,v_{2g}$ of $L$ such that for $i<j$ we have $\theta(v_i,v_j)=1$ if $(i,j) \in \{(1,2),(3,4),\ldots\}$ and $0$ otherwise.

\begin{problem}\label{p:higherdim1} For $g>1$,
characterize the pairs $(q_1,q_2)$ of integral quadratic forms in $2g$ variables so that there exists a symplectic rank $2g$ direct summand $L$ of $\Z^{4g}$ such that 
$q_L$ is $\mathrm{Sp}_{2g}(\Z)$-equivalent to $q_1$ and $q_{L^\pperp}$ is $\mathrm{Sp}_{2g}(\Z)$-equivalent to $q_2$.
Here, as before $L^\pperp= \{v \in \Z^{4g}:\theta(v,w) = 0 \text{ for all }w \in L\}$.
\end{problem}

As phrased, we suspect Problem~\ref{p:higherdim1} to be difficult, though various weakenings of it appear amenable, for instance, to equidistribution methods.
For instance, one can ask for which discriminants a pair $(q_1,q_2)$ exists.

\begin{example}
Pairs of forms $(q_1,q_2)$ as in Problem~\ref{p:higherdim1} may be constructed as follows: 
suppose that $(q_{1,i},q_{2,i})$ for $1\leq i \leq g$ are pairs of IBQFs so that $[a_ix^2+xy+c_iy^2]^2*[q_{1,i}]=\overline{[q_{2,i}]}$ for some $a_i,c_i \in \Z$. By the present article, the pairs
$(q_{1,i},q_{2,i})$ may be realized in the sense of Problem~\ref{p:higherdim1} for $g=1$.
It is then easy to see that one may realize the pair $(q_{1,1}\oplus \dots \oplus q_{1,g},q_{2,1}\oplus \dots \oplus q_{2,g})$.
For a concrete example, one can set
\begin{align}\label{eq:Hayden}
q_1 = \sum_{i=1}^g (x_i^2 + x_iy_i + 6y_i^2),\quad
q_2 = \sum_{i=1}^g (2x_i^2 + x_iy_i + 3y_i^2).
\end{align}
Observe that $q_1,q_2$ as in \eqref{eq:Hayden} are not in the same $\GL_{2g}(\Z)$-class because $q_2$ does not represent $1$ (see also Example~\ref{ex:-23}).
We remark that a prototypical pair as in Problem~\ref{p:higherdim1} should not be of the special form considered in this example.
\end{example}

\subsection{Topology questions}
Livingston's question~\cite{Livingston82} ``Do there exist knots with two Seifert surfaces of the same genus $g$ that remain non-isotopic when pushed into the $4$-ball?'' is resolved for all $g\geq 1$.
\begin{example}
The example of $q_1$ and $q_2$ in~\eqref{eq:Hayden} was pointed out to the authors by Kyle Hayden for its topological significance. Let $F_1$ and $F_2$ be any pair of genus $g$ Seifert surfaces for the same knot realizing $q_1$ and $q_2$ as their Seifert form. (Such exist by the $g\geq 1$ analogue of Lemma~\ref{lem:realizationofL} or, to be concrete, one may take $F_1$ and $F_2$ as the $g$-fold boundary connected sum of $G_1$ and the $g$-fold boundary connected sum of $G_2$ respectively, where $G_1$ and $G_2$ are disjoint genus 1 pairs that have $x^2 + xy + 6y^2$ and $2x^2 + xy + 3y^2$ as their Seifert forms.) Then $\Sigma(F_1^{\B^4})$ and $\Sigma(F_2^{\B^4})$ have non-isometric intersection forms and, in particular, the pushed-in surfaces $F_1^{\B^4}$ and $F_2^{\B^4}$ are not topologically ambiently isotopic.
\end{example}

We wonder whether more subtle versions of Livingston's question can also be answered affirmatively. 
For example, fixing $g\geq 1$ and a determinant $D\in \Z$, one may wonder, whether there exist genus $g$ Seifert surfaces $F_1$ and $F_2$ for the same knot with determinant $D$ such that 
 $F_1^{\B^4}$ and $F_2^{\B^4}$ are not topologically ambiently isotopic. We pose a more refined question. 
\begin{problem}\label{p:higherdim2}
Consider bilinear forms $U:\Z^{2g}\times \Z^{2g}\to \Z$ that antisymmetrize to a form of determinant $1$. 
For which such $U$ do there exist genus $g$ Seifert surfaces $F_1$ and $F_2$ for the same knot such that the Seifert form of $F_1$ is isometric to $U$ and the oriented $4$-manifolds $\Sigma(F_1^{\B^4})$ and $\Sigma(F_2^{\B^4})$ have non-isometric intersection form?
\end{problem}

Problem~\ref{p:higherdim2} with the additional requirement of $F_1$ and $F_2$ being disjoint in their interior is equivalent to asking which quadratic forms $q$ arise as one of the entries of a pair as in Problem~\ref{p:higherdim1}.

Let us return to genus $g=1$.
For which pairs $s_1$ and $s_2$ of (equivalence classes of) IBQFs of the same discriminant $D \equiv 1 \mod 4$ is there a knot $K$ with two genus one Seifert surfaces $F_1$, $F_2$ having Seifert forms $s_1$, $s_2$ respectively?  We have answered this question completely under the additional assumption that $F_1$ and $F_2$ are disjoint in Theorem~\ref{thmintro:main3d}.  
Without this restriction, we have the necessary condition that $F_1, F_2$ must be S-equivalent, but we do not know if this is sufficient. In the context of this paper, S-equivalence of binary quadratic forms is quickest described by saying that $s_1$ and $s_2$ are \emph{S-equivalent} if there exists an $s\in\mathcal{S}^+_D$ such that $s*s_1=s_2$. 
Hence, the relevant remaining question, which is equivalent to the conjecture stated in \S\ref{ss:Sequiv}, is the following.

\begin{problem}\label{p:S-equi<->realization}
Fix an integer $D$ such that $D\equiv 1 \pmod{4}$, and let $s\in\mathcal{S}^+_D$ and $s_1,s_2\in\mathcal{Q}^+_D$ be such that $s*s_1=s_2$.
Do there exist genus $1$ Seifert surfaces $F_1$ and $F_2$ with boundary the same knot such that $s_1=s_{F_1}$ and $s_2=s_{F_2}$?
\end{problem}

Given that in this text, we obstruct ambient isotopy of $F_1^{\B^4}$ and $F_2^{\B^4}$ in $\B^4$ using that $s_{F_1}\notin\{s_{F_2},\overline{s_{F_2}}\}$ are different, the following is a natural problem to consider. 
\begin{problem}\label{p:distinguishbeyondSF}
Do there exist genus $1$ Seifert surfaces $F_1$ and $F_2$ with boundary the same knot such that $s_{F_1}=s_{F_2}$ and $F_1^{\B^4}$ and $F_2^{\B^4}$ are not topologically ambiently isotopic in $\B^4$?
\end{problem}

\bibliographystyle{amsalpha}
\bibliography{Bibliography}

\end{document}

%% file: preamble.tex
\usepackage[utf8]{inputenc}
\usepackage{amssymb,amsmath,mathtools,latexsym, amscd, epsfig}
\usepackage{amsthm}
\usepackage{setspace}
\usepackage{graphicx}
\usepackage{mathrsfs}
\usepackage{mdwlist}
\usepackage{color}
\usepackage{bbm}
\usepackage{url}
\usepackage[colorlinks=true]{hyperref}
\usepackage{breakurl}
\usepackage[all,cmtip]{xy}
\usepackage[shortlabels]{enumitem}
\usepackage{tikz-cd}

%% file: environments.tex
\newtheorem{theorem}{Theorem}
\newtheorem{lem}[theorem]{Lemma}
\newtheorem{definition}[theorem]{Definition}
\newtheorem{proposition}[theorem]{Proposition}
\newtheorem{prop}[theorem]{Proposition}

\newtheorem{problem}[theorem]{Problem}
\newtheorem{cor}[theorem]{Corollary}
\theoremstyle{remark}\newtheorem{remark}[theorem]{Remark}
\newtheorem{example}[theorem]{Example}

\numberwithin{theorem}{section}
\numberwithin{equation}{section}

%% file: macros.tex
\makeatletter
\newsavebox{\@brx}
\newcommand{\llangle}[1][]{\savebox{\@brx}{\(\m@th{#1\langle}\)}%
  \mathopen{\copy\@brx\kern-0.5\wd\@brx\usebox{\@brx}}}
\newcommand{\rrangle}[1][]{\savebox{\@brx}{\(\m@th{#1\rangle}\)}%
  \mathclose{\copy\@brx\kern-0.5\wd\@brx\usebox{\@brx}}}
\makeatother

\newcommand{\GL}{\operatorname{GL}}
\newcommand{\PGL}{\operatorname{PGL}}
\newcommand{\SL}{\operatorname{SL}}

\newcommand{\Stab}{\operatorname{Stab}} 

\newcommand{\quat}{\operatorname{M}_2} 
\newcommand{\trace}{\operatorname{Tr}} 
\newcommand{\quatzero}{\operatorname{M}_2^0}

\newcommand{\disc}{\operatorname{disc}}
\renewcommand{\mod}{\, \mathrm{mod}\,} 

\newcommand{\rank}{\operatorname{rank}}

\newcommand\rquot[2]{
  \mathchoice
  {
    \text{\raise0.5ex\hbox{$#1$}\big/\lower0.5ex\hbox{$#2$}}%
  }
  {
    #1\,/\,#2
  }
  {
    #1\,/\,#2
  }
  {
    #1\,/\,#2
  }
}

\newcommand\lquot[2]{
  \mathchoice
  {
    \text{\lower0.5ex\hbox{$#1$}\big\backslash\raise0.5ex\hbox{$#2$}}%
  }
  {
    #1\,\backslash\,#2
  }
  {
    #1\,\backslash\,#2
  }
  {
    #1\,\backslash\,#2
  }
}

\newcommand\lrquot[3]{
  \mathchoice
  {
    \text{\lower0.5ex\hbox{$#1$}\big\backslash\raise0.5ex\hbox{$#2$}\big/
      \lower0.5ex\hbox{$#3$}}%
  }
  {
    #1\,\backslash\,#2\,/\,#3
  }
  {
    #1\,\backslash\,#2\,/\,#3
  }
  {
    #1\,\backslash\,#2\,/\,#3
  }
}

\newcommand{\Bcal}{\mathcal{B}}

\newcommand{\Kcal}{\mathcal{K}}
\newcommand{\Lcal}{\mathcal{L}}

\newcommand{\bQ}{\mathbb{Q}}
\newcommand{\Q}{\mathbb{Q}}

\newcommand{\R}{\mathbb{R}}

\newcommand{\bZ}{\mathbb{Z}}
\newcommand{\Z}{\mathbb{Z}}

\newcommand{\set}[1]{\left\{#1\right\}}
\newcommand{\pa}[1]{\left(#1\right)}
\newcommand{\SmMat}[4]{
\big(\begin{smallmatrix}
  #1 & #2\\ #3 & #4
\end{smallmatrix}\big)
}

\newcommand{\pperp}{{\perp\kern -4.2pt \perp}}

\newcommand{\Notj}{j}

\newcommand{\B}{\mathbf{B}} 
 
\newcommand{\GrossLat}{\Lambda} 
\newcommand{\Id}{\operatorname{Id}} 

\newcommand{\id}{\mathrm{id}}

\newcommand{\content}{\operatorname{content}}

%% file: surfaces.pdf_tex
\begingroup%
  \makeatletter%
  \providecommand\color[2][]{%
    \errmessage{(Inkscape) Color is used for the text in Inkscape, but the package 'color.sty' is not loaded}%
    \renewcommand\color[2][]{}%
  }%
  \providecommand\transparent[1]{%
    \errmessage{(Inkscape) Transparency is used (non-zero) for the text in Inkscape, but the package 'transparent.sty' is not loaded}%
    \renewcommand\transparent[1]{}%
  }%
  \providecommand\rotatebox[2]{#2}%
  \newcommand*\fsize{\dimexpr\f@size pt\relax}%
  \newcommand*\lineheight[1]{\fontsize{\fsize}{#1\fsize}\selectfont}%
  \ifx\svgwidth\undefined%
    \setlength{\unitlength}{372.64595572bp}%
    \ifx\svgscale\undefined%
      \relax%
    \else%
      \setlength{\unitlength}{\unitlength * \real{\svgscale}}%
    \fi%
  \else%
    \setlength{\unitlength}{\svgwidth}%
  \fi%
  \global\let\svgwidth\undefined%
  \global\let\svgscale\undefined%
  \makeatother%
  \begin{picture}(1,0.17181003)%
    \lineheight{1}%
    \setlength\tabcolsep{0pt}%
    \put(0,0){\includegraphics[width=\unitlength,page=1]{surfaces.pdf}}%
    \put(0.16201068,0.02655997){\color[rgb]{0,0,0}\makebox(0,0)[lt]{\lineheight{1.25}\smash{\begin{tabular}[t]{l}$\textcolor{red}{\beta_1}$\end{tabular}}}}%
    \put(0.32844727,0.02946525){\color[rgb]{0,0,0}\makebox(0,0)[lt]{\lineheight{1.25}\smash{\begin{tabular}[t]{l}$\textcolor{red}{\beta_2}$\end{tabular}}}}%
    \put(0.12287676,0.13302262){\color[rgb]{0,0,0}\makebox(0,0)[lt]{\lineheight{1.25}\smash{\begin{tabular}[t]{l}$\textcolor{red}{\alpha_1}$\end{tabular}}}}%
    \put(0.28587072,0.1323029){\color[rgb]{0,0,0}\makebox(0,0)[lt]{\lineheight{1.25}\smash{\begin{tabular}[t]{l}$\textcolor{red}{\alpha_2}$\end{tabular}}}}%
    \put(0.7290558,0.03227588){\color[rgb]{0,0,0}\makebox(0,0)[lt]{\lineheight{1.25}\smash{\begin{tabular}[t]{l}$\alpha$\end{tabular}}}}%
    \put(0.82722906,0.0534835){\color[rgb]{0,0,0}\makebox(0,0)[lt]{\lineheight{1.25}\smash{\begin{tabular}[t]{l}$\beta$\end{tabular}}}}%
    \put(0.61587847,0.11302866){\color[rgb]{0,0,0}\makebox(0,0)[lt]{\lineheight{1.25}\smash{\begin{tabular}[t]{l}$\textcolor{red}{\beta^+}$\end{tabular}}}}%
    \put(0,0){\includegraphics[width=\unitlength,page=2]{surfaces.pdf}}%
  \end{picture}%
\endgroup%

%% file: 9_5.pdf_tex
\begingroup%
  \makeatletter%
  \providecommand\color[2][]{%
    \errmessage{(Inkscape) Color is used for the text in Inkscape, but the package 'color.sty' is not loaded}%
    \renewcommand\color[2][]{}%
  }%
  \providecommand\transparent[1]{%
    \errmessage{(Inkscape) Transparency is used (non-zero) for the text in Inkscape, but the package 'transparent.sty' is not loaded}%
    \renewcommand\transparent[1]{}%
  }%
  \providecommand\rotatebox[2]{#2}%
  \newcommand*\fsize{\dimexpr\f@size pt\relax}%
  \newcommand*\lineheight[1]{\fontsize{\fsize}{#1\fsize}\selectfont}%
  \ifx\svgwidth\undefined%
    \setlength{\unitlength}{409.74176842bp}%
    \ifx\svgscale\undefined%
      \relax%
    \else%
      \setlength{\unitlength}{\unitlength * \real{\svgscale}}%
    \fi%
  \else%
    \setlength{\unitlength}{\svgwidth}%
  \fi%
  \global\let\svgwidth\undefined%
  \global\let\svgscale\undefined%
  \makeatother%
  \begin{picture}(1,0.35550717)%
    \lineheight{1}%
    \setlength\tabcolsep{0pt}%
    \put(0,0){\includegraphics[width=\unitlength,page=1]{9_5.pdf}}%
    \put(0.05613407,0.06721552){\color[rgb]{0,0,0}\makebox(0,0)[lt]{\lineheight{1.25}\smash{\begin{tabular}[t]{l}$\textcolor{red}{\beta}$\end{tabular}}}}%
    \put(0.62353687,0.06754275){\color[rgb]{0,0,0}\makebox(0,0)[lt]{\lineheight{1.25}\smash{\begin{tabular}[t]{l}$\textcolor{red}{\beta}$\end{tabular}}}}%
    \put(0.90037745,0.23574154){\color[rgb]{0,0,0}\makebox(0,0)[lt]{\lineheight{1.25}\smash{\begin{tabular}[t]{l}$\textcolor{red}{\alpha}$\end{tabular}}}}%
    \put(0.4045473,0.18478802){\color[rgb]{0,0,0}\makebox(0,0)[lt]{\lineheight{1.25}\smash{\begin{tabular}[t]{l}$\textcolor{red}{\alpha}$\end{tabular}}}}%
    \put(0,0){\includegraphics[width=\unitlength,page=2]{9_5.pdf}}%
  \end{picture}%
\endgroup%